\documentclass[11pt,reqno]{amsart}
\usepackage[tmargin=1in,bmargin=1in,rmargin=1in,lmargin=1in]{geometry}

\usepackage{amsmath,amsthm,amssymb,enumerate,mathrsfs,color,tocvsec2}
\usepackage[breaklinks=true]{hyperref}
\allowdisplaybreaks

\theoremstyle{plain}

\newtheorem{theorem}[equation]{Theorem}
\newtheorem{cor}[equation]{Corollary}
\newtheorem{prop}[equation]{Proposition}
\newtheorem{lemma}[equation]{Lemma}
\newtheorem{question}[equation]{Question}
\newtheorem{utheorem}{\textrm{\textbf{Theorem}}}

\theoremstyle{definition}

\newtheorem{remark}[equation]{Remark}
\newtheorem{defn}[equation]{Definition}

\newtheorem*{defn2}{Definition}

\numberwithin{equation}{section}

\newcommand{\R}{\mathbb{R}}
\newcommand{\Z}{\mathbb{Z}}
\newcommand{\bp}{\mathbb{P}}
\newcommand{\bx}{\mathbf{x}}
\newcommand{\by}{\mathbf{y}}
\newcommand{\bc}{\mathbf{c}}
\newcommand{\inc}[2]{{#1}^{#2, \uparrow}}
\newcommand{\TN}{\mathrm{TN}}
\newcommand{\TP}{\mathrm{TP}}
\newcommand{\kjk}{K_{\mathcal{JKS}}}
\newcommand{\comment}[1]{}
\renewcommand{\geq}{\geqslant}
\renewcommand{\leq}{\leqslant}

\begin{document}
\vspace*{-1mm}

%{{{1 Frontmatter
\title[Multiply positive functions, critical exponents,
Jain--Karlin--Schoenberg kernel\ \ ]{Multiply positive functions,
critical exponent phenomena, and the Jain--Karlin--Schoenberg kernel}

\author{Apoorva Khare}

\address{Department of Mathematics, Indian Institute of Science; and
Analysis \& Probability Research Group; Bangalore -- 560012, India}

\email{\tt khare@iisc.ac.in}

\date{\today}



\begin{abstract}
This paper continues the analysis of multiply positive functions, first
studied by Schoenberg in [\textit{Ann.\ of Math.} 1955].
We prove the converse to a result of Karlin [\textit{Trans.\ Amer.\
Math.\ Soc.} 1964], and also strengthen his result and two results of
Schoenberg [\textit{Ann.\ of Math.} 1955]. One of the latter results
concerns zeros of Laplace transforms of multiply positive functions. The
other results study which powers $\alpha$ of two specific kernels are
totally non-negative of order $p \geq 2$ (denoted TN$_p$); both authors
showed this happens for $\alpha \geq p-2$, and Schoenberg proved that it
does not for $\alpha < p-2$. We show more strongly that for every $p
\times p$ submatrix of either kernel, up to a `shift', its $\alpha$th
power is totally positive of order $p$ (TP$_p$) for every $\alpha > p-2$,
and is not TN$_p$ for every $\alpha \in (0,p-2) \setminus \mathbb{Z}$. 
We also extend Karlin's result to a larger class of non-smooth P\'olya
frequency functions. In particular, these results reveal `critical
exponent' phenomena in the theory of total positivity. 
%(The same critical exponent $(p-2)$ was first discovered by
%FitzGerald--Horn in [\textit{J.\ Math.\ Anal.\ Appl.} 1977] for positive
%semidefiniteness.)
We also prove the converse to a 1968 result of Karlin, revealing yet
another critical exponent phenomenon -- for Laplace transforms of all
P\'olya frequency functions.

We further classify the powers preserving all Hankel TN$_p$ kernels on
intervals, and isolate individual kernels encoding these powers; the
latter strengthens a result in previous joint work in [\textit{J.\ Eur.\
Math.\ Soc.}, in press].
We then transfer results on preservers by P\'olya--Szeg\H{o} (1925),
Loewner/Horn [\textit{Trans.\ Amer.\ Math.\ Soc.} 1969], and joint with
Tao [\textit{Amer.\ J.\ Math.}, in press]
from positive semidefinite matrices to Hankel TN$_p$ kernels. An
additional application is to construct individual matrices that encode
the Loewner convex powers. This complements Jain's results [\textit{Adv.\
in Oper.\ Th.}\ 2020] for Loewner positivity, which we strengthen to
total positivity, with self-contained proofs. Remarkably, these
(strengthened) results of Jain, those of Schoenberg and Karlin, the
latter's converse, and the aforementioned individual Hankel kernels all
arise from a single symmetric rank-two kernel and its powers:
$\max(1+xy,0)$.

In addition, we provide a novel characterization of P\'olya frequency
functions and sequences of order $p \geq 3$, following Schoenberg's
result for $p=2$ in [\textit{J.\ d'Analyse Math.} 1951]. We also correct
a small gap in that same paper, in Schoenberg's classification of the
discontinuous P\'olya frequency functions.
\end{abstract}
\maketitle

%In this paper, we prove in a self-contained manner, both old results by
%FitzGerald--Horn, Jain, and others (including recently in our joint
%works) -- as well as new ones such as converses and analogues of results
%by Karlin. The results in the literature deal with entrywise power
%functions preserving positivity or total non-negativity on matrices; in
%our revisiting many of them as well as in proving new results, we are
%able to show total positivity. This strengthens several old results
%simultaneously. In addition, we explain how all of these results, old
%and new, arise from studying the powers of a \textit{single} kernel,
%which we term the \textit{Jain--Karlin--Schoenberg kernel}.

\settocdepth{section}
\tableofcontents
\newpage
%}}}

\noindent \textit{Notation:}
\begin{enumerate}
\item A \textit{positive semidefinite matrix} is a real symmetric matrix
with non-negative eigenvalues. Given $I \subset \R$ and $n \geq 1$,
denote the space of such $n \times n$ matrices with entries in $I$ by
$\bp_n(I)$.

\item The \textit{Loewner ordering} on $\R^{n \times n}$ is the partial
order where $M \geq N$ if and only if $M-N \in \bp_n$.

\item Following Schur~\cite{Schur1911}, a function $f : I \to \R$ acts
\textit{entrywise} on $\bp_n(I)$ via: $f[A] := (f(a_{jk}))_{j,k=1}^n$.

\item We say that a map $f : I \to \R$ preserves \textit{Loewner
positivity} on $\bp_n(I)$ if $f[A] \geq 0$ for all $A \in \bp_n(I)$,
i.e., for $A \geq 0$.

\item We will adopt the convention $0^0 := 0$, unless otherwise
specified.
\end{enumerate}

\begin{defn2}
Let $X,Y$ be totally ordered sets, and $p \geq 1$ an integer. 
\begin{enumerate}
\item Define $\inc{X}{p}$ to be the set of all $p$-tuples $\bx = (x_1,
\dots, x_p) \in X$ with strictly increasing coordinates: $x_1 < \cdots <
x_p$. (In his book~\cite{Karlin}, Karlin denotes this open simplex by
$\Delta_p(X)$.)

\item A kernel $K : X \times Y \to \R$ is \textit{totally non-negative of
order $p$}, denoted $\TN_p$, if for all integers $1 \leq r \leq p$ and
tuples $\bx \in \inc{X}{r}, \by \in \inc{Y}{r}$, the determinant of the
matrix
\[
K[ \bx; \by ] := (K(x_j, y_k))_{j,k=1}^r
\]
is non-negative. We say $K$ is \textit{totally non-negative ($\TN$)} if
$K$ is $\TN_p$ for all $p \geq 1$.

\item Analogously, one defines $\TP_p$ and $\TP$ kernels. If the domains
$X,Y$ are both finite, then this yields $\TN_p, \TN, \TP_p$, or $\TP$
matrices.
\end{enumerate}
\end{defn2}

%{{{1 Section 1 - Introduction and main results
\section{Introduction and main results}\label{S1}

In recent joint works~\cite{BGKP-hankel,BGKP-TN}, we explored the
preservers of various classes of positive semidefinite, $\TN$, and $\TP$
kernels on infinite domains -- as well as the preservers of $\TN_p$ and
$\TP_p$ kernels on finite domains. The present paper studies preservers
of $\TN_p$ kernels, albeit on infinite domains -- this was initiated by
Schoenberg in 1955~\cite{Schoenberg55}. In doing so, we end up bringing
under this roof, several old and new results on powers preserving Loewner
positivity, monotonicity, and convexity as well.

Positive semidefinite matrices, totally positive ($\TP$) matrices, and
operations preserving these structures have been widely studied in the
literature. More generally, the same question applies to post-composition
operators applied to (structured) kernels with the various notions of
positivity. 
P\'olya frequency functions~\cite{Schoenberg51} and sequences~\cite{Fe}
constitute important classes of totally non-negative ($\TN$) kernels that
have been widely studied in
analysis~\cite{Schoenberg51,Schoenberg55},
interpolation theory~\cite{Curry2},
differential equations and integrable systems~\cite{KW-2,Loewner55},
probability and statistics~\cite{Johnstone,Efron,Karlin}, 
and combinatorics~\cite{Br1} (to name a few areas and a very few
sources). More generally, $\TN$ and $\TP$ matrices occur in multiple
areas of mathematics, ranging from the aforementioned fields to
representation theory and flag varieties~\cite{L-1,Ri2},
cluster algebras~\cite{BZ-2,FZ-3},
interacting particle systems~\cite{GK,GK1},
and Gabor analysis~\cite{GRS,GS}.
We refer the reader to the twin surveys~\cite{BGKP-survey1,BGKP-survey2}
and references therein -- specifically, to the comprehensive book of
Karlin~\cite{Karlin} -- for more on $\TN/\TP$ matrices and kernels.

\subsection{The critical exponent $n-2$ in positivity}

A well-studied theme in the matrix positivity literature involves
entrywise real powers acting on matrices (say with positive entries), to
preserve positive (semi)definiteness or other Loewner properties. This
theme owes its origins to Loewner, who was interested in understanding
(in connection with the Bieberbach conjecture) which entrywise powers
preserve positive semidefiniteness. This was resolved by FitzGerald and
Horn:

\begin{theorem}[FitzGerald and Horn, 1977,
\cite{FitzHorn}]\label{Tfitzhorn}
Let $n \geq 2$ be an integer and $\alpha \in \R$.
\begin{enumerate}
\item The entrywise map $x^\alpha$ preserves Loewner positivity on
$\bp_n((0,\infty))$ if and only if $\alpha \in \Z^{\geq 0} \cup
[n-2,\infty)$.
\item The entrywise map $x^\alpha$ preserves Loewner monotonicity on
$\bp_n((0,\infty))$ if and only if $\alpha \in \Z^{\geq 0} \cup
[n-1,\infty)$. Here, we say a map $f : I \to \R$ is \textit{Loewner
monotone} on $\bp_n(I)$ if $f[A] \geq f[B]$ whenever $A \geq B$ in
$\bp_n(I)$.
\end{enumerate}
\end{theorem}

This phase transition at $\alpha = n-2$ for positivity preservers (resp.\
$\alpha = n-1$ for monotonicity preservers) is known as a
\textit{critical exponent} in the matrix analysis literature.
See~\cite{JW} for a survey of the early history of this phenomenon. More
recently, a plethora of papers have studied Loewner positive entrywise
powers on the domain $I = (0,\infty)$ or $\R$, and on test sets of
positive matrices constrained by rank and sparsity~\cite{Bhatia-Elsner,
GKR-crit-2sided, GKR-critG, Hiai2009, Jain, Jain2}. These have yielded
similar critical exponents (including a `combinatorial' one for every
graph~\cite{GKR-critG,GKR-fpsac}).

In fact the earliest occurrence of this critical exponent $(n-2)$ -- in
the positive semidefiniteness literature -- was in Horn's 1969
article~\cite{horn}. Horn began with an important result of Loewner on
continuous maps preserving Loewner positivity on $\bp_n((0,\infty))$
(which remains essentially the only known necessary condition to date,
for such maps in fixed dimension) -- see Theorem~\ref{Thorn}. From this,
Horn deduced the `only if' part of Theorem~\ref{Tfitzhorn}(1): for
$\alpha \in (0,n-2) \setminus \Z$, there exists a matrix $A_\alpha \in
\bp_n((0,\infty))$ such that $A_\alpha^{\circ \alpha}$ is not positive
semidefinite. Horn's proof was non-constructive; moreover, such a
`counterexample' matrix $A_\alpha$ would \textit{a priori} depend on
$\alpha$, as is also the case in the proof of
Theorem~\ref{Tfitzhorn}(1),(2). This dependence was recently removed, as
we explain presently.

\subsection{The critical exponent $n-2$ in total positivity}

At almost the same time\footnote{In fact, also at the same place
(Stanford University); Karlin, Loewner, P\'olya, and Szeg\H{o} had been
colleagues, and FitzGerald and Horn were Loewner's students.} as Horn's
aforementioned article containing Loewner's result, Karlin had completed
his important and comprehensive monograph~\cite{Karlin} on total
positivity. One can find in it the same set of powers as above -- now
acting on a certain P\'olya frequency function. In this case, however,
Karlin showed (originally in his 1964 paper~\cite{KarlinTAMS}) the
`reverse' direction to Horn above:

\begin{theorem}[{Karlin, \cite{KarlinTAMS} -- see also~\cite[Ch.~4, \S
4, p.~211]{Karlin}}]\label{Tkarlin}
Let $p \geq 2$ be an integer and $\alpha \geq 0$. Define the P\'olya
frequency function
\begin{equation}
\Omega(x) := \begin{cases}
x e^{-x}, \qquad & \text{if } x > 0,\\
0, &\text{otherwise}.
\end{cases}
\end{equation}
If $\alpha \in \Z^{\geq 0} \cup [p-2,\infty)$, then the function
$\Omega(x)^\alpha$ is $\TN_p$.
\end{theorem}

In particular, for every integer $\alpha > 0$, the function
$\Omega^\alpha$ is a P\'olya frequency function -- this was originally
shown by Schoenberg in 1951~\cite{Schoenberg51}. We explain the notation
used here and in the sequel:

\begin{defn}
Let $p \geq 1$ be an integer, and $\Lambda : \R \to \R$ a Lebesgue
measurable function.
\begin{enumerate}
\item We say $\Lambda$ is a \textit{P\'olya frequency function} if
$\Lambda$ is Lebesgue integrable on $\R$, the associated Toeplitz kernel
\[
T_\Lambda : \R \times \R \to \R, \qquad (x,y) \mapsto \Lambda(x-y)
\]
is totally non-negative, and $\Lambda$ does not vanish at least at two
points (whence on an interval).

\item We say $\Lambda$ is \textit{totally non-negative of order $p \geq
1$}, again denoted $\TN_p$, if $T_\Lambda$ is $\TN_p$. If $\Lambda$ is
$\TN_p$ for all $p \geq 1$, then we say $\Lambda$ is \textit{totally
non-negative ($\TN$)}.

\item Analogously, one defines $\TP_p$ and $\TP$ functions.
\end{enumerate}
\end{defn}

Karlin's result is at least the second instance of a critical exponent
phenomenon, implicit in the theory of total positivity. Almost a decade
earlier, Schoenberg had shown a similar result for powers of a seemingly
unrelated kernel, which he termed \textit{Wallis distributions}:

\begin{theorem}[{Schoenberg, 1955,~\cite[Theorems~4
and~5]{Schoenberg51}}]\label{Tschoenberg}
Define the map
\begin{equation}
W : \R \to \R, \qquad x \mapsto \begin{cases}
\cos(x), \qquad & \text{if}\ x \in (-\pi/2,\pi/2),\\
0, & \text{otherwise}.
\end{cases}
\end{equation}
Also suppose $\alpha \geq 0$ and an integer $p \geq 2$. Then
$W(x)^\alpha$ is $\TN_p$ if and only if $\alpha \geq p-2$.
\end{theorem}

\noindent (The `only if' part implicitly follows from~\cite[Theorem
4]{Schoenberg55} and was not formulated. For completeness, we write out
how this can be achieved, in Remark~\ref{Rschoenberg}.) Thus,
Schoenberg's result shows a critical exponent phenomenon from total
positivity -- with the same point $p-2$ for a $\TN_p$ kernel, as for
positivity preservers on $p \times p$ matrices.

%It turns out that Karlin's kernel $T_\Omega$ is intimately related to its
%precursor: Schoenberg's `Wallis/cosine kernel' $T_W$. As we explain
%below, $T_\Omega$ is a `rescaling' of the restriction to $\R \times
%(0,\infty)$ of a distinguished kernel $\kjk$, and $T_W$ can be obtained
%from $\kjk$ using a $\TN$/$\TP$-preserving transform and `rescaling'.
%We also show how this provides an alternate proof of Karlin's
%theorem~\ref{Tkarlin}, via two results by Schoenberg: a simple argument
%in~\cite{Schoenberg51} (see the Appendix) and Theorem~\ref{Tschoenberg}
%above.

In parallel: note that Karlin did not address the non-integer powers
below $p-2$. We begin by achieving this task, and showing that $\alpha =
p-2$ is indeed a `critical exponent' for total positivity:

\begin{theorem}\label{T1}
Let $p \geq 2$ be an integer and $\alpha \in (0,p-2) \setminus \Z$. Then
$\Omega^\alpha$ is not $\TN_p$.
\end{theorem}

One consequence is that there also exists a sequence of \textit{P\'olya
frequency sequences}\footnote{Recall, P\'olya frequency sequences are
defined to be real sequences ${\bf a} = (a_n)_{n \in \Z}$ such that the
Toeplitz kernel $T_{\bf a} : \Z \times \Z \to \R$ sending $(m,n) \mapsto
a_{m-n}$ is $\TN$.}
whose $\alpha$th powers are not $\TN_p$ for $\alpha \in (0,p-2) \setminus
\Z$. This follows from the continuity of the kernel $\Omega$, via a
discretization argument as in our recent joint work~\cite{BGKP-TN}. The
assertion can be strengthened to show the existence of $\TN$ Toeplitz
kernels on more general domains $X \times Y$ than $\Z \times \Z$. These
subsets $X,Y$ only need to satisfy: for each $p \geq 1$, there exist
equi-spaced arithmetic progressions $\bx \in \inc{X}{p}$ and $\by \in
\inc{Y}{p}$ with $x_2 - x_1 = y_2 - y_1$. A similar argument works for
Schoenberg's powers $W(x)^\alpha$.

Thus, Theorems~\ref{Tschoenberg} and~\ref{T1} say that for each $\alpha
\in (0, p-2) \setminus \Z$, one can find tuples $\bx, \by \in
\inc{\R}{p}$ (or in $\inc{\mathbb{Q}}{p}$ via discretization), for which
the Toeplitz matrices $T_{\Omega^\alpha}[ \bx; \by ]$ and $T_{W^\alpha}[
\bx; \by ]$ each contain a negative minor. Our first main result
strengthens both of these conditions, by showing they are satisfied up to
a shift at \textit{every} pair $\bx, \by$, and simultaneously for
\textit{all} powers $\alpha \in (0, p-2) \setminus \Z$:

\begin{utheorem}\label{ThmA}
Fix an integer $p \geq 2$ and subsets $X,Y \subset \R$ of size at least
$p$.
\begin{enumerate}
\item There exists $a = a(X,Y) \in \R$ such that the restriction of
$T_{\Omega_a}(x,y)^\alpha$ to $X \times Y$ (where $\Omega_a(x) =
\Omega(x-a)$ as in Theorem~\ref{Tkarlin}), is not $\TN_p$ for all $\alpha
\in (0,p-2) \setminus \Z$. 

\item There exists $m = m(X,Y) \in (0,\infty)$ such that the restriction
of $T_{W_m}(x,y)^\alpha$ to $X \times Y$ (where $W_m(x) = W(mx)$ as in
Theorem~\ref{Tschoenberg}), is not $\TN_p$ for all $\alpha \in (0,p-2)
\setminus \Z$. 

\item Given tuples $\bx, \by \in \inc{\R}{p}$, there exist $a \in \R$ and
$m > 0$ such that the matrices
\[
(\Omega(x_j - y_k - a)^\alpha)_{j,k=1}^p, \qquad
(W(m(x_j - y_k))^\alpha)_{j,k=1}^p
\]
are $\TP$ if $\alpha > p-2$, $\TN$ if $\alpha \in \{ 0, 1, \dots, p-2
\}$, and not $\TN$ if $\alpha \in (0,p-2) \setminus \Z$.
\end{enumerate}
\end{utheorem}

Note that the additive/multiplicative shifts $a = a(\bx, \by)$ and $m =
m(\bx, \by)$ are independent of $\alpha \in (0, p-2) \setminus \Z$. Hence
so are $a(X,Y), m(X,Y)$.

\begin{remark}
The first two assertions in Theorem~\ref{ThmA}(3) strengthen Karlin's
theorem~\ref{Tkarlin} and one implication in Schoenberg's
theorem~\ref{Tschoenberg}, on a suitable part of their domains. The final
assertion in Theorem~\ref{ThmA}(3) is the aforementioned strengthening of
the `converse' Theorem~\ref{T1} (and of the other implication in
Theorem~\ref{Tschoenberg}), and follows from parts~(1) and~(2) by
specializing $X,Y$ to the sets of coordinates of $\bx, \by$ respectively.
\end{remark}

Theorems~\ref{T1} and~\ref{ThmA} lead to a P\'olya frequency function
whose non-integer powers are not $\TN$:

\begin{cor}\label{C11}
If $\alpha \geq 0$ and the function $\Omega(x)^\alpha$ is $\TN$, then
$\alpha$ is an integer.
\end{cor}

This was observed e.g.~in~\cite{BGKP-TN}, where the `heavy machinery' of
the bilateral Laplace transform was used through deep results of
Schoenberg~\cite{Schoenberg51}.\footnote{Briefly, the bilateral Laplace
transform of $\Omega^\alpha$ is $\Gamma(\alpha+1)/(s+\alpha)^{\alpha+1}$,
and if $\alpha \not\in \Z^{\geq 0}$ then its reciprocal is not analytic
in $s$ -- not in the Laguerre--P\'olya class. Thus $\Omega^\alpha$ is not
a P\'olya frequency function by~\cite{Schoenberg51}, whence not $\TN$.}
Our proof of Theorem~\ref{ThmA} below is self-contained, shows a stronger
result, and avoids these sophisticated tools.

Thus, our first contribution shows that critical exponents for total
non-negativity -- more strongly, `total positivity' phenomena -- occur in
the study of preservers of P\'olya frequency functions, P\'olya frequency
sequences, and Toeplitz kernels on more general domains, in the above
(strengthened) results by Schoenberg and Karlin and their converses --
and for all submatrices, up to a shift.

\begin{remark}
A direct application of Theorem~\ref{ThmA} yields yet another critical
exponent phenomenon in total positivity -- this time on on the level of
Laplace transforms of P\'olya frequency functions. One implication can be
found in Karlin's book, and the converse follows from our results. See
Corollary~\ref{Cthma}.
\end{remark}

\begin{remark}
Via the Laplace transform, we also show in Section~\ref{Shw} that
Karlin's result and Theorem~\ref{ThmA} are `degenerate' cases of a more
general phenomenon. Namely, if $\Omega^{(q,r)}(x)$ is the (unique)
P\'olya frequency function with bilateral Laplace transform $\mathcal{B}
\{ \Omega^{(q,r)} \} (s) = (1 + s/q)^{-1} (1 + s/r)^{-1}$ for $q \neq r$
positive scalars, then we show that $\Omega^{(q,r)}(x)^\alpha$ satisfies
the same results as does $\Omega(x)^\alpha$ above. The `degenerate' case
of $q = r = 1$ precisely yields $\Omega(x)$ and Theorem~\ref{ThmA}.
\end{remark}

\subsection{Single-matrix encoders; Hankel kernels}

As seen above, Schoenberg and Karlin studied individual kernels, for
which all powers $\geq p-2$ preserve $\TN_p$, and no non-integer power $<
p-2$ does so -- in close analogy with the FitzGerald--Horn
theorem~\ref{Tfitzhorn}. In the latter, parallel setting of entrywise
powers preserving positivity, such individual matrices were discovered
only recently, by Jain~\cite{Jain,Jain2}. Her results are now stated in
parallel to Theorem~\ref{Tfitzhorn}, and isolate a smallest possible test
set for Loewner positive and monotone powers:

\begin{theorem}[Jain, 2020, \cite{Jain2}]\label{Tjain}
Let $n \in \Z, n \geq 2$ and $\alpha \in \R$. Suppose $x_1, \dots, x_n
\in \R$ are pairwise distinct, with $1 + x_j x_k > 0\ \forall j,k$. Let
$A := (1 + x_j x_k)_{j,k=1}^n$ and $B := {\bf 1}_{n \times n}$, so
$A \geq B \geq 0$.
\begin{enumerate}
\item The matrix $A^{\circ \alpha}$ is positive semidefinite if and only
if $\alpha \in \Z^{\geq 0} \cup [n-2,\infty)$.

\item Suppose all $x_j$ are non-zero. The matrix $A^{\circ \alpha} \geq B
= B^{\circ \alpha}$, if and only if $\alpha \in \Z^{\geq 0} \cup
[n-1,\infty)$.
\end{enumerate}
\end{theorem}

In fact Jain does more in~\cite{Jain,Jain2}: she computes the inertia of
the matrices $A^{\circ \alpha}$ as above, for all real $\alpha \geq 0$.
Our main theorem~\ref{ThmC} below strengthens Theorem~\ref{Tjain}(1),
and shows that $A^{\circ \alpha}$ is not just positive definite for
$\alpha > n-2$, but totally positive. In particular, as can be shown
using Perron's theorem~\cite{Perron} and the folklore theorem of
Kronecker on eigenvalues of compound matrices, $A^{\circ \alpha}$ has
simple, positive eigenvalues for $\alpha > n-2$, parallel to Jain.

Before proceeding further, we describe two consequences of the first part
of Jain's theorem~\ref{Tjain}:
\begin{enumerate}
\item Set $x_j := \cot(j \pi/(2n))$; now $A^{\circ \alpha}$ is positive
semidefinite if and only if so is the matrix
\[
D^{\circ \alpha} A^{\circ \alpha} D^{\circ \alpha} = (D A D)^{\circ
\alpha},
\]
where $D$ is the diagonal matrix with $(j,j)$ entry $\sin(j \pi / (2n))$.
But $D A D$ is the Toeplitz matrix $(\cos((j-k)\pi/(2n)))_{j,k=1}^n$, so
Jain's result yields a rank-two positive semidefinite Toeplitz matrix
which encodes the Loewner positive powers on $\bp_n((0,\infty))$. Notice
this is a restriction of Schoenberg's kernel $T_W$ from
Theorem~\ref{Tschoenberg}.

\item Setting $x_j := u_0^j$ for $u_0 \in (0,\infty) \setminus \{ 1 \}$,
it follows that $A$ is a rank-two positive semidefinite Hankel matrix,
which encodes the Loewner positive and monotone powers on
$\bp_n((0,\infty))$.
\end{enumerate}

This second consequence leads to our next theorem. Recall that Karlin and
Schoenberg's results above, together with Theorem~\ref{ThmA}, studied
Toeplitz kernels which encoded the (non-integer) powers preserving
$\TN_p$. We next produce a Hankel kernel with this property.
Unfortunately, the naive guess of $K(x,y) = (x+y) e^{-(x+y)}$ does not
work, since this is `equivalent' to $T_\Omega(x,-y)$, which leads to
`row-reversal' and hence a sign of $(-1)^{p(p-1)/2}$ in $p \times p$
submatrices drawn from $K$. (As a specific instance, $\det
T_\Omega[(3,4);(-2,-1)] < 0$.)
However, the `rank-two' kernel $1 + u_0^{x+y}$ is $\TN$ and exhibits the
same critical exponent phenomenon. More strongly, this kernel encodes the
powers preserving $\TN_p$ for \textit{all} Hankel kernels on $\R \times
\R$ -- in other words, the analogues of Theorems~\ref{Tfitzhorn}
and~\ref{Tjain} hold together, for Hankel kernels on $\R \times \R$.
Slightly more strongly, this happens over arbitrary intervals:

\begin{utheorem}\label{ThmB}
Let $p \geq 2$ be an integer, and fix scalars $c_0, u_0 > 0, u_0 \neq 1$
and $\alpha \geq 0$. Also fix an interval $X_0 \subset \R$ with positive
measure. The following are equivalent:
\begin{enumerate}
\item If $X \subset \R$ is an interval with positive measure, and $H : X
\times X \to \R$ is a continuous $\TN_p$ Hankel kernel, then $H^\alpha$
is $\TN_p$. Here, by a Hankel kernel we mean $K : X \times X \to \R$ such
that there exists a function $f : X+X \to \R$ satisfying: $K(x,y) =
f(x+y)$ for $x,y \in X$.

\item Define the Hankel kernel
\[
H_{u_0} : X_0 \times X_0 \to \R, \qquad (x,y) \mapsto 1 + c_0 u_0^{x+y}.
\]
Then $H_{u_0}^\alpha$ is $\TN_p$ on $X_0 \times X_0$.

\item $\alpha \in \Z^{\geq 0} \cup [p-2,\infty)$.
\end{enumerate}
In particular, every $\alpha \in \Z^{\geq 0}$ preserves $\TN$ Hankel
kernels. Moreover, for every $\bx, \by \in \inc{X_0}{p}$, the kernel
$H_{u_0}^\alpha$ is $\TP_p$ if $\alpha > p-2$, and not $\TN_p$ if $\alpha
\in (0,p-2) \setminus \Z$. 
\end{utheorem}

This strengthens results in recent work \cite{BGKP-hankel,BGKP-TN}, which
study powers preserving $\TN$ Hankel kernels. Theorem~\ref{ThmB} studies
power preservers of $\TN_p$ Hankel kernels, for each $p \geq 2$.

\subsection{The Jain--Karlin--Schoenberg kernel}

Our next main result again concerns power-preservers of $\TN_p$ kernels.
We show that remarkably, the multitude of kernels studied above are all
related. More precisely, Karlin's theorem~\ref{Tkarlin} and our converse,
Schoenberg's theorem~\ref{Tschoenberg}, the FitzGerald--Horn
theorem~\ref{Tfitzhorn}, Jain's theorem~\ref{Tjain}(1), the
aforementioned strengthenings of these, the Hankel kernels $H_{u_0}$, and
the related critical exponent phenomena \textit{all} arise from studying
a particular symmetric kernel having `rank two' (on part of its domain)
-- restricted to various sub-domains. In particular, this will explain
why the same critical exponent of $p-2$ (plus, all powers above $p-2$,
and no non-integer power below it) shows up in each of these settings.

We begin by introducing this simple kernel:

\begin{defn}
Define the \textit{Jain--Karlin--Schoenberg kernel} $\kjk$ as follows:
\begin{equation}
\kjk : \R \times \R \to \R, \qquad (x,y) \mapsto \max(1+xy,0).
\end{equation}
\end{defn}

The choice of name is because -- as we explain in Remark~\ref{Rmaster} --
the restrictions of this kernel to $(-\infty,0] \times (0,\infty)$, to
$(0,\infty) \times (0,\infty)$, and on the full domain $\R^2$, are
intimately related to Karlin's kernel $\Omega$, to Jain's matrices $(1 +
x_j x_k)$, and to Schoenberg's cosine-kernel $W$, respectively.

Our next result studies the powers of $\kjk$ that are $\TN_p$ on the
plane or on the $X$ or $Y$ half-planes. Remark~\ref{Rmaster} will then
explain how this connects to all of the results stated above.

\begin{utheorem}\label{ThmC}
Fix an integer $p \geq 2$, an interval $I \subset \R$, and let a scalar
$\alpha \geq 0$.
\begin{enumerate}
\item $\kjk^\alpha$ is $\TN_p$ on $\R \times \R$ for $\alpha \geq p-2$.

\item If the power $\kjk^\alpha$ is $\TN_p$, then $\alpha \in \Z^{\geq 0}
\cup [p-2,\infty)$. More strongly, given $\bx, \by \in \inc{\R}{p}$ such
that $1 + x_j y_k > 0\ \forall j,k$, the matrix $\kjk[ \bx; \by ]^{\circ
\alpha}$ is:
\begin{enumerate}
\item $\TP$ if $\alpha > p-2$;
\item $\TN$ if $\alpha \in \{ 0, 1, \dots, p-2 \}$; and
\item not $\TN$ if $\alpha \in (0, p-2) \setminus \Z$.
\end{enumerate}

\item Suppose $I \subset [0,\infty)$ or $I \subset (-\infty,0]$. The
kernel $\kjk^\alpha$ is $\TN_p$ on $I \times \R$ (or $\R \times I$) if
and only if $\alpha \in \Z^{\geq 0} \cup [p-2,\infty)$. (In particular,
$\kjk^\alpha$ is $\TN$ on $I \times \R$ or $\R \times I$ for $\alpha \in
\Z^{\geq 0}$.)
\end{enumerate}
\end{utheorem}

As an aside, integer powers of the kernel $\kjk$ (more precisely, of $1 +
xy$) have featured in the statistics and machine learning literature, as
non-homogeneous polynomial kernels of dot-product type. See
e.g.~\cite{HSS,HL,KLL,SOW,SSM}.

Our final two results deal with
(a)~more general $\TN_p$ functions than powers, and (b)~general kernels.
A closely related result to Theorem~\ref{ThmC} is a 1955 theorem by
Schoenberg~\cite{Schoenberg55}, which implies that no power $\alpha <
p-2$ of the kernel $W$ is $\TN_p$. This is a result on arbitrary
compactly supported, multiply positive functions $\Lambda$, and we
strengthen it by restricting the domain of $\Lambda$:

\begin{utheorem}\label{ThmD}
Suppose $0 < \rho \leq \widetilde{\rho} \leq +\infty$ and $0 < \epsilon
\leq \widetilde{\rho} - \rho/2$
are scalars, with $\rho < \infty$. Suppose $p \geq 2$ is an integer, and
the integrable function $\Lambda : (-\widetilde{\rho},\widetilde{\rho})
\to \R$ is positive on $(-\rho/2,\rho/2)$, vanishes outside $[-\rho/2,
\rho/2]$, and induces the $\TN_p$ kernel
\[
T_\Lambda : [0,\epsilon) \times (-\rho/2, (\rho/2) + \epsilon) \to \R,
\qquad (x,y) \mapsto \Lambda(x-y).
\]
Then the Fourier--Laplace transform
\[
\mathcal{B} \{ \Lambda \}(s) := \int_{-\rho/2}^{\rho/2} e^{-sx}
\Lambda(x)\ dx, \qquad s \in \mathbb{C}
\]
has no zeros in the strip $|\Im(s)| < p \pi / \rho$.
\end{utheorem}

Schoenberg proved this result in~\cite{Schoenberg55}, assuming
$\widetilde{\rho} = +\infty$ and that $T_\Lambda : \R \times \R \to \R$
is $\TN_p$. (He also `changed variables' so that $\rho = \pi$.)
This means that all minors of order $\leq p$ drawn from $\Lambda$ are
required to be non-negative. We arrive at the same conclusions as
Schoenberg, using far fewer minors -- indeed, the aforementioned domain
of $T_\Lambda$ means that we only need to work with the restriction of
$\Lambda$ to $(-(\rho/2) - \epsilon, (\rho/2) + \epsilon)$ for
arbitrarily small $\epsilon > 0$.

Our final result provides a characterization of $\TN_p$ functions (or
P\'olya frequency functions of order $p$). Recall that such a result was
shown for $p=2$ by Schoenberg in 1951~\cite{Schoenberg51}, and Weinberger
mentioned in 1983 a variant for $p=3$ in~\cite{Weinberger83} (which turns
out to have a small gap). To our knowledge, no such characterization is
known for $p \geq 4$. This is provided by the next result, by considering
only the largest-sized minors:

\begin{utheorem}\label{ThmE}
Let $p \geq 3$ be an integer, and a function $\Lambda : \R \to
[0,\infty)$. The following are equivalent.
\begin{enumerate}
\item Either $\Lambda(x) = e^{ax+b}$ for $a,b \in \R$, or:
(a)~$\Lambda$ is Lebesgue measurable;
(b)~for all scalars $x_0, y_0$, the function $\Lambda(x_0  - y) \Lambda(y
- y_0) \to 0$ as $y \to \infty$; and
(c)~$\det T_\Lambda[\bx; \by] \geq 0$ for all $\bx,\by \in \inc{\R}{p}$.

\item The function $\Lambda : \R \to \R$ is $\TN_p$.
\end{enumerate}
\end{utheorem}

The result also holds for $p=2$, in which case it is a tautology.
The proof-technique also yields similar results for `P\'olya frequency
sequences of order $p$' -- or more generally, for (not necessarily
Toeplitz) $\TN_p$ kernels on $X \times Y$ for general subsets $X,Y
\subset \R$ -- under similar decay assumptions. See the final section of
the paper.

\subsection*{Organization of the paper}

The next section develops a few preliminaries -- specifically, novel
homotopy arguments that are used in our proofs. The subsequent five
sections of the paper prove our main theorems, one per section.
The first three of these sections contain three other features:

(a)~After proving Theorem~\ref{ThmA}, we show akin to Jain's
theorem~\ref{Tjain} that the same `individual' matrices $(1 + x_j
x_k)_{j,k=1}^n$ and ${\bf 1}_{n \times n}$ also encode the powers
preserving \textit{Loewner convexity}. See Section~\ref{Sconvex} for the
definition of Loewner convexity as well as the precise result.

(b)~After proving Theorem~\ref{ThmB}, we present results -- now for
Hankel $\TN_p$ kernel preservers -- parallel to Loewner's aforementioned
necessary condition in~\cite{horn}, to an old observation of
P\'olya--Szeg\H{o}~\cite{polya-szego}, and to our recent work with
Tao~\cite{KT} on polynomial preservers of positivity on $p \times p$
matrices.

(c)~After proving Theorem~\ref{ThmC}, we explain in Remark~\ref{Rmaster}
how this result for $\kjk$ subsumes our results above, as well as results
of Karlin, Jain, and Schoenberg.

The Appendix, which may safely be skipped during a first reading,
contains proofs of several results pertaining to power-preservers of
$\TN_p$, and is included for the convenience of the reader.

\section{A variant of Descartes' rule of signs, and homotopy arguments}

The proofs of the above results rely on new tools and old. We begin with
a variant from~\cite{Jain2} of Descartes' rule of signs, in which
exponentials are replaced by powers $(1 + ux_j)^r$. To state this result
requires the following notation.

\begin{defn}
Given an integer $n \geq 1$ and a tuple $\bx = (x_1, \dots, x_n) \in
\R^n$, define
\[
A_\bx := \begin{cases}
-\infty, \qquad & \text{if } \max_j x_j \leq 0,\\
-1/\max_j x_j, & \text{otherwise},
\end{cases} \qquad
B_\bx := \begin{cases}
\infty, \qquad & \text{if } \min_j x_j \geq 0,\\
-1/\min_j {x_j}, & \text{otherwise}.
\end{cases}
\]
\end{defn}

\begin{prop}[Jain, \cite{Jain2}]\label{Pjain}
Fix an integer $n \geq 1$ and real tuples ${\bf c} = (c_1, \dots,
c_n) \neq 0$ and $\bx = (x_1, \dots, x_n)$, where the $x_j$ are pairwise
distinct. For a real number $r$, define the function
\[
\varphi_{\bx,\bc,r} : (A_\bx, B_\bx) \to \R, \qquad u \mapsto
\sum_{j=1}^n c_j (1 + u x_j)^r.
\]
Then either $\varphi_{\bx,\bc,r} \equiv 0$, or it has at most $n-1$
zeros, counting multiplicities.
\end{prop}

In the interest of keeping this paper self-contained, we sketch this
proof in a somewhat vestigial Appendix, together with proofs of the
results from other works that are used in this paper.

The next step is a (novel) homotopy argument for symmetric matrices; a
non-symmetric variant will also be proved and used below.

\begin{prop}\label{Phomotopy}
Fix an integer $n \geq 2$ and real scalars
\[
x_1 < \cdots < x_n \quad \text{and} \quad 0 < y_1 < \cdots < y_n,
\quad \text{with} \quad 1 + x_j x_k > 0\ \forall j,k.
\]
There exists $\delta > 0$ such that for all $0 < \epsilon \leq \delta$,
the `linear homotopies' between $x_j$ and $\epsilon y_j$, given by
\[
x_j^{(\epsilon)}(t) := x_j + t (\epsilon y_j - x_j), \qquad t \in [0,1]
\]
satisfy
\[
1 + x_j^{(\epsilon)}(t) x_k^{(\epsilon)}(t) > 0, \qquad \forall 1 \leq
j,k \leq n, \ t \in [0,1].
\]
\end{prop}

\begin{remark}\label{Rjain}
The above result is (implicitly stated, and explicitly) used
in~\cite{Jain2} with all $y_j = j$, and without the factor of $\epsilon$.
The use of this result is key if one wishes to avoid using Jain's prior
work~\cite{Jain} in proving Theorem~\ref{Tjain}. Unfortunately, the
factor of $\epsilon$ here is crucial, otherwise the result fails to hold.
Here are two explicit examples; in both of them, $n = 2$, $\epsilon = 1$,
and $(y_1, y_2) = (1,2)$.
Suppose first that $(x_1, x_2) = (-199,0)$; then `completing the square'
shows that the above assertion fails to hold at `most' times in the
homotopy:
\[
1 + x_1^{(1)}(t) x_2^{(1)}(t) \leq 0, \quad \forall t \in \left[
\frac{398}{800} - \frac{1}{20} \sqrt{ \frac{398^2}{40^2} - 1},\
\frac{398}{800} + \frac{1}{20} \sqrt{ \frac{398^2}{40^2} - 1} \right],
\]
and this interval contains $[0.0026, 0.9924]$.
As another example, if $(x_1, x_2) = (-8.5,0.1)$, then
\[
1 + x_1^{(1)}(t) x_2^{(1)} (t) \leq 0, \quad \forall t \in \left[
\frac{8 - \sqrt{61}}{19},\ \frac{8 + \sqrt{61}}{19} \right] \supset
[0.01, 0.8321].
\]
\end{remark}

\begin{remark}
Jain has communicated to us~\cite{Jain3} a short workaround to the above
gap in~\cite{Jain2}, as follows: if all $x_j \leq 0$ then to prove
Theorem~\ref{Tjain}(1) one can replace all $x_j$ with $-x_j$. If $x_1 < 0
< x_n$ then one lets $0 < y_1 < \cdots < y_n < x_n$, and for these
\textit{specific} $y_j$, the homotopy argument works. However, we then
need to show Theorem~\ref{Tjain}(1) in the special case when all $x_j >
0$ -- which is a result in Jain's prior work; see~\cite{Jain} and the
references and results cited therein. These prior results involve
strictly sign regular (SSR) matrices and earlier papers. In this paper we
avoid SSR matrices, and hence our approach additionally serves to provide
a shorter, direct proof of Theorem~\ref{Tjain}.
\end{remark}

We now show the above homotopy result.

\begin{proof}[Proof of Proposition~\ref{Phomotopy}]
We make three clarifying observations to start the proof, with $x_j(t)$
denoting $x_j^{(\epsilon)}(t)$ throughout for a fixed $\epsilon > 0$.
First, the assumptions imply $x_1(t) < \cdots < x_n(t)$ for all $t \in
[0,1]$.

Second, if $x_1 = x_1(0) \geq 0$, then clearly $x_j(t) \geq 0$ for all $t
\in [0,1]$ and all $1 \leq j \leq n$, and in this case the result follows
at once. We will thus assume in the sequel that $x_1 < 0$.

Third, suppose there exist integers $1 \leq j < k \leq n$ and a time $t
\in [0,1]$ such that $1 + x_j(t) x_k(t) \leq 0$, then we have $x_j(t) < 0
< x_k(t)$, and so $x_1(t) < 0 < x_n(t)$. A straightforward computation
shows
\[
1 + x_1(t) x_n(t) \leq 1 + x_j(t) x_k(t) \leq 0.
\]

Given these observations, suppose we have initial data $x_j, y_j$, with
$x_1 < 0$ from above. It suffices to find $\delta > 0$ such that
\[
1 + x_1^{(\epsilon)}(t) x_n^{(\epsilon)}(t) > 0, \qquad \forall \epsilon
\in (0,\delta], \ t \in (0,1).
\]
Depending on the sign of $x_n$, we consider two cases:\medskip

\noindent \textbf{Case 1:} $x_n \geq 0$, in which case $x_n < 1/|x_1|$.
We claim that $\delta := 1 / (|x_1| y_n)$ works. Indeed, given $0 <
\epsilon \leq \delta$, and $t \in (0,1)$, compute:
\begin{align*}
1 + x_1^{(\epsilon)}(t) x_n^{(\epsilon)}(t) = &\ 1 + (t \epsilon y_1 + (1
- t) x_1) (t \epsilon y_n + (1 - t) x_n)\\
> &\ 1 + (1-t) x_1 (t \epsilon y_n + (1 - t) x_n)\\
> &\ 1 + (1-t) x_1 (t \epsilon y_n + (1 - t)/ |x_1|),
\end{align*}
with both inequalities strict because $t \in (0,1)$.
Now the final expression equals
\[
= 1 - (1-t)^2 + t (1-t) \epsilon y_n x_1
\geq t \left( 2-t - (1-t) \delta y_n |x_1| \right) = t > 0.
\]

\noindent \textbf{Case 2:} $x_n < 0$. Define the continuous function
\[
g(\epsilon) := 1 - \frac{\epsilon^2 (x_n y_1 - x_1 y_n)^2}{4(\epsilon y_1
- x_1) (\epsilon y_n - x_n)}, \qquad \epsilon \geq 0.
\]
Since $g(0) > 0$, there exists $\delta > 0$ such that $g$ is positive on
$[0, \delta]$.
We claim this choice of $\delta$ works. Fix $0 < \epsilon \leq \delta$,
and define
\[
t_j^{(\epsilon)} := -x_j / (\epsilon y_j - x_j), \qquad \forall j \in
[1,n].
\]
It is easy to check that $x_j^{(\epsilon)}(t)$ is positive, zero, or
negative when $t > t_j^{(\epsilon)}, t = t_j^{(\epsilon)}, t <
t_j^{(\epsilon)}$ respectively; moreover, since all $x_j < 0$, the above
observations imply
\[
0 < t_n^{(\epsilon)} < t_{n-1}^{(\epsilon)} < \cdots < t_1^{(\epsilon)} <
1.
\]
In particular, if $0 \leq t \leq t_n^{(\epsilon)}$ or $t_1^{(\epsilon)}
\leq t \leq 1$, then $x_1^{(\epsilon)}(t)$ and $x_n^{(\epsilon)}(t)$ both
have the same sign, whence $1 + x_1^{(\epsilon)}(t) x_n^{(\epsilon)}(t)
\geq 1$, so is positive. Otherwise $t_n^{(\epsilon)} < t <
t_1^{(\epsilon)}$, in which case we first note that
\[
x_j^{(\epsilon)}(t) = t \epsilon y_j + (1 - t) x_j = (t -
t_j^{(\epsilon)}) (\epsilon y_j - x_j), \qquad \forall j \in [1,n], \ t
\in [0,1].
\]
But now we compute, using the AM--GM inequality and choice of $\delta$:
\begin{align*}
1 + x_1^{(\epsilon)}(t) x_n^{(\epsilon)}(t) = &\ 1 + (t -
t_1^{(\epsilon)}) (t - t_n^{(\epsilon)}) (\epsilon y_1 - x_1) (\epsilon
y_n - x_n)\\
\geq &\ 1 - \frac{1}{4} (t_1^{(\epsilon)} - t_n^{(\epsilon)})^2 (\epsilon
y_1 - x_1) (\epsilon y_n - x_n) = g(\epsilon) > 0. \qedhere
\end{align*}
\end{proof}

Our next result is more widely applicable, at the cost of making the
homotopy `piecewise linear':

\begin{prop}\label{Phomotopy2}
Fix an integer $n \geq 2$ and tuples of real scalars
\[
\bx, \by, {\bf p}, {\bf q} \in \inc{\R}{n}
\]
such that $1 + x_j y_k > 0\ \forall j,k$ and $p_1, q_1 > 0$. Then there
exists piecewise linear homotopies
\[
x_j(t), y_j(t) : [0,1] \to \R, \qquad 1 \leq j \leq n
\]
such that $\bx(t), \by(t) \in \inc{\R}{n}$ for all times $t \in [0,1]$,
with
\[
x_j(0) = x_j, \quad x_j(1) = p_j, \quad y_j(0) = y_j, \quad x_j(1) = q_j,
\]
and such that $1 + x_j(t) y_k(t) > 0$ for all $t \in [0,1]$.
\end{prop}

\begin{proof}
Let $\delta_1 := \frac{1}{2 |y_1| p_n}$ if $y_1 \neq 0$, and $1$
otherwise. Define
\[
x'_j(t) := x_j + t (\delta_1 p_j - x_j), \qquad 1 \leq j \leq n, \ t \in
[0,1].
\]
We claim that $1 + x'_j(t) y_k > 0$ for all $1 \leq j,k \leq n$ and $t
\in [0,1]$. This is true at $t=0$ for all $j,k$; now suppose it fails for
some $t_0 \in (0,1]$ and $j,k \in [1,n]$. If $y_k \geq 0$ then $0 >
x'_j(t_0) \geq x_j$, so
\[
0 \geq 1 + x'_j(t_0) y_k \geq 1 + x_j y_k > 0,
\]
which is impossible. Thus we must have
\[
y_k < 0 < x'_j(t_0) \leq x'_n(t_0) \leq \max(\delta_1 p_n, x_n).
\]
Using this,
\[
0 \geq 1 + x'_j(t_0) y_k \geq 1 + x'_j(t_0) y_1 \geq 1 + y_1
\max(\delta_1 p_n, x_n) = \min(1 + y_1 x_n, 1 + \delta_1 y_1 p_n) > 0,
\]
which is similarly impossible.

This reasoning shows that one can define a linear homotopy $\bx(t), \ t
\in [0,1/3]$ going from $\bx$ to $\delta_1 {\bf p}$ for some $\delta_1 >
0$, such that $1 + x_j(t) y_k > 0$ for all $t$. Throughout, we define
$\by(t) \equiv \by$ for $t \in [0,1/3]$.

In a similar fashion, we let $\bx(t) \equiv \delta_1 {\bf p}$ for $t \in
[1/3, 2/3]$, and write down a linear homotopy $\by(t)$ from $\by$ to
$\delta_2 {\bf q}$ for some $\delta_2 > 0$, such that $1 + x_j(t) y_k(t)
> 0$ for $t \in [1/3,2/3]$.

Finally, let $\bx(t)$ (respectively $\by(t)$) for $t \in [2/3,1]$ be the
linear homotopy from $\delta_1 {\bf p}$ to ${\bf p}$ (respectively from
$\delta_2 {\bf q}$ to ${\bf q}$). Since $p_1, q_1 > 0$, it is trivially
true that $1 + x_j(t) y_k(t) > 0$ for $t \in [2/3,1]$.
\end{proof}
%}}}

%{{{1 Section 3 - Proof of Theorem~\ref{ThmA} and its strengthening:
%Critical exponent for PF functions
\section{Proof of Theorem~\ref{ThmA} and its strengthening: Critical
exponent for PF functions}

We now show the main results above. The next step is a direct application
of Proposition~\ref{Pjain}:

\begin{prop}[Jain,~\cite{Jain2}]\label{Pjain2}
Suppose $x_1, \dots, x_n \in \R$ are pairwise distinct, as are $y_1,
\dots, y_n \in \R$. If $1 + x_j y_k > 0$ for all $j,k$, and $\alpha \in
\R \setminus \{ 0, 1, \dots, n-2 \}$, then $S^{\circ \alpha}$ is
non-singular, where $S := (1 + x_j y_k)_{j,k=1}^n$. If $\alpha \in \{ 0,
1, \dots, n-2 \}$, then $S^{\circ \alpha}$ has rank $\alpha+1$.
\end{prop}

\noindent Once again, the short proof is outlined in the Appendix.

In this and later sections, we provide applications of
Proposition~\ref{Pjain2}: to our main theorems, as well as to Jain's
theorem~\ref{Tjain}. All of these applications also rely on the (novel)
homotopy argument in Proposition~\ref{Phomotopy}; this keeps the proofs
in this paper self-contained. We begin with Theorem~\ref{Tjain}, as it is
used in the subsequent proofs.

\begin{proof}[Proof of Theorem~\ref{Tjain}]\hfill
\begin{enumerate}
\item If $\alpha \in \Z^{\geq 0} \cup [n-2,\infty)$ then $A^{\circ
\alpha} \in \bp_n$ by Theorem~\ref{Tfitzhorn}(1) (the proof of which is
outlined in the Appendix). We will need a refinement of the converse
result, so we sketch this argument, taken from~\cite{FitzHorn}. The
result is easily shown for $\alpha < 0$, so we suppose $\alpha \in
(0,n-2) \setminus \Z$ -- in particular, $n \geq 3$ now. Let $\bx =
\bx(\epsilon) := \epsilon (1,2,\dots,n)^T$ with $\epsilon > 0$, and
choose any vector $v \in \R^n$ that is orthogonal to ${\bf 1}, \bx,
\bx^{\circ 2}, \dots, \bx^{\circ (\lfloor \alpha \rfloor + 1)}$ but not
to $\bx^{\circ (\lfloor \alpha \rfloor + 2)}$. (Here, $\bx^{\circ m} =
\epsilon^m (1, \dots, n^m)^T$ for an integer $m$.) Now using binomial
series, one computes:
\[
v^T ({\bf 1}_{n \times n} + \bx \bx^T)^{\circ \alpha} v = \epsilon^{2
(\lfloor \alpha \rfloor + 2)} \binom{\alpha}{\lfloor \alpha \rfloor + 2}
(v^T \bx^{\circ (\lfloor \alpha \rfloor + 2)})^2 + o(\epsilon^{2 (\lfloor
\alpha \rfloor + 2)}).
\]
Divide by $\epsilon^{2 (\lfloor \alpha \rfloor + 2)}$ and let $\epsilon
\to 0^+$; as the right-hand side has a negative limit, the matrix-power
on the left cannot be positive semidefinite.

With this special case at hand, the general case follows, via a more
direct argument than in~\cite{Jain,Jain2}. Given pairwise distinct $x_j$
such that $1 + x_j x_k > 0\ \forall j,k$, let $y_j := \epsilon j$, where
$\epsilon > 0$ is small enough to satisfy both the argument in the
preceding paragraph, as well as the conclusions of
Proposition~\ref{Phomotopy}. Now let $x_j(t) := x_j + t (\epsilon j -
x_j)$ and let $C(t) := (1 + x_j(t) x_k(t))^{\circ \alpha}$. Then the
smallest eigenvalue $\lambda_{\min}(C(1)) < 0$ from above, and $C(t)$ is
always non-singular by Proposition~\ref{Pjain2}. It follows by the
continuity of eigenvalues (or a simpler, direct argument) that
$\lambda_{\min}(C(0)) < 0$, as desired.

\item We show the `if' part of Theorem~\ref{Tfitzhorn}(2)
from~\cite{FitzHorn} for self-completeness (and also because it is used
presently). If $\alpha \in \Z^{\geq 0}$ and $C \geq D \geq 0$ in
$\bp_n((0,\infty))$, then
\[
C^{\circ \alpha} \geq C^{\circ (\alpha - 1)} \circ D \geq \cdots \geq
D^{\circ \alpha},
\]
by the Schur product theorem.\footnote{The Schur product
theorem~\cite{Schur1911} says that if $A,B \in \bp_n(\R)$, then so is
their entrywise product $A \circ B := (a_{jk} b_{jk})_{j,k=1}^n$. (For
self-completeness: This is easily checked using the spectral
eigen-decompositions of $A,B$.)}
If $\alpha \geq n-1$, then by the fundamental theorem of calculus,
\[
C^{\circ \alpha} - D^{\circ \alpha} = \alpha \int_0^1 (C - D) \circ
(\lambda C + (1-\lambda)D)^{\circ (\alpha-1)}\ d \lambda.
\]
By Theorem~\ref{Tfitzhorn}(1) and the Schur product theorem, the
integrand is positive semidefinite, whence we are done.

The `only if' part of Theorem~\ref{Tfitzhorn}(2) follows from
Theorem~\ref{Tjain}(2), which is immediate from the preceding part:
Suppose $A^{\circ \alpha} \geq B = B^{\circ \alpha}$, with $A = (1 + x_j
x_k)_{j,k=1}^n$ and $B = {\bf 1}$ as given. If $\bx' := (\bx^T, 0)^T \in
\R^{n+1}$, then the matrix
\[
\widetilde{A} := {\bf 1}_{(n+1) \times (n+1)} + \bx' (\bx')^T =
\begin{pmatrix} A & {\bf 1} \\ {\bf 1}^T & 1 \end{pmatrix}
\]
satisfies the hypotheses of part~(1). Using Schur complements and
part~(1), we thus have:
\[
A^{\circ \alpha} \geq {\bf 1}_{n \times n} \quad \Longleftrightarrow
\quad \widetilde{A}^{\circ \alpha} \in \bp_{n+1} \quad
\Longleftrightarrow \quad \alpha \in \Z^{\geq 0} \cup [n-1,\infty).
\qedhere
\]
\end{enumerate}
\end{proof}

This concludes a self-contained (modulo the Appendix) proof of
Theorem~\ref{Tjain}, avoiding the use of SSR matrices as
in~\cite{Jain,Jain2} (see Remark~\ref{Rjain}).
A key corollary, used repeatedly below, now strengthens
Theorem~\ref{Tjain} from positive (semi)definiteness to total positivity,
as promised above:

\begin{cor}\label{Cdet}
Let $p \geq 2$ be an integer, and $\bx, \by \in \inc{\R}{p}$ be tuples
such that $1 + x_j y_k > 0$ for all $j,k$. Let the matrix $C := (1 + x_j
y_k)_{j,k=1}^p$.
\begin{enumerate}
\item If $\alpha > p-2$ then $C^{\circ \alpha}$ is $\TP$.

\item If $\alpha \in \{ 0, 1, \dots, p-2 \}$, then $C^{\circ \alpha}$ has
rank $\alpha + 1$.

\item If $\alpha \in (0,p-2) \setminus \Z$, then $C^{\circ \alpha}$ is
not $\TN$ -- in fact, it has a principal minor that is negative.
\end{enumerate}
\end{cor}

See also Corollary~\ref{Cssr} below, for a stronger version with more
detailed information. (This corollary is not required in the present
work, but we include it below for completeness.)

\begin{proof}
The second part follows from Proposition~\ref{Pjain2}. For the third, fix
any tuple ${\bf q} \in \inc{(0,\infty)}{p}$, and use
Proposition~\ref{Phomotopy2} to construct piecewise linear homotopies
$\bx(t), \by(t), \ t \in [0,1]$ from $\bx, \by$ to ${\bf q}$
respectively, such that $1 + x_j(t) y_k(t) > 0$ for all $1 \leq j,k \leq
p$ and $t \in [0,1]$. Let $C(t) := {\bf 1}_{p \times p} + \bx(t)
\by(t)^T$. Then $C(1)^{\circ \alpha} = ({\bf 1}_{p \times p} + {\bf q}
{\bf q}^T)^{\circ \alpha}$ is not positive semidefinite by
Theorem~\ref{Tjain}(1), hence has a negative principal minor. Now use
Proposition~\ref{Phomotopy2} to show that the same principal minor of
$C(0)^{\circ \alpha} = C^{\circ \alpha}$ is negative, again by
Proposition~\ref{Pjain2}.

For the first part, let $B$ be any square submatrix of $C$ of order $p'
\in [2,p]$; then one can repeat the preceding argument with $C = B$ for
this part. Thus, let ${\bf q}$ and $C(t)_{p' \times p'} = B(t)$ be as in
the previous paragraph. Since now $\alpha > p'-2$, so $\det B(t)^{\circ
\alpha}$ does not change sign, by Proposition~\ref{Pjain2}. But $\det
B(1)^{\circ \alpha} > 0$ by Theorem~\ref{Tjain}(1). This shows that every
minor of the original matrix $C_{p \times p}^{\circ \alpha}$ is positive,
whence $C^{\circ \alpha}$ is $\TP$.
\end{proof}

With this and the preceding ingredients at hand, we show our first main
result.

\begin{proof}[Proof of Theorem~\ref{ThmA}]\hfill
\begin{enumerate}
\item We first show the result for $X = Y = \R$. Notice in this case that
the result for any $a \in \R$ shows the result for any other, so we work
with $a=0$. Suppose $\alpha \in (0,p-2) \setminus \Z$, and
\[
0 < v_p < \cdots < v_1 < u_p < \cdots < u_1
\]
are fixed scalars. Set $x_j := u_j^{-1}$ and $y_k := -v_k$; thus $x_j >
0$ and $1 + x_j y_k > 0$ for all $j,k$. By (the proof of)
Corollary~\ref{Cdet}(3), the matrix $C := ((1 + x_j
y_k)^\alpha)_{j,k=1}^p$ has a negative minor, hence is not $\TN$. Pre-
and post-multiply by diagonal matrices with $(j,j)$ entry $u_j^\alpha
e^{-\alpha u_j}$ and $e^{\alpha v_j}$ respectively. This shows, via
applying the order-reversing permutation to the rows and to the columns,
that given
\[
{\bf u}' := (u_1', \dots, u_p'),\ {\bf v}' := (v_1', \dots, v_p') \in
\inc{\R}{p}, \quad \text{with} \quad v_p' < u_1',
\]
the matrix $T_{\Omega^\alpha}[ {\bf u}'; {\bf v}' ] = T_{\Omega^\alpha}[
{\bf u}' - (v'_1 - 1) {\bf 1}; {\bf v}' - (v'_1 - 1) {\bf 1}]$ has a
submatrix with negative determinant.

This shows the result when $X = Y = \R$, e.g.~with $a=0$. For arbitrary
$X,Y \subset \R$ of sizes at least $p$, first choose and fix increasing
$p$-tuples ${\bf u}' \in \inc{X}{p}, {\bf v}' \in \inc{Y}{p}$; now choose
any $a < u_1' - v_p'$. By the above proof, the matrix
$T_{\Omega_a(x)^\alpha}[ {\bf u}'; {\bf v}'] = T_{\Omega^\alpha}[ {\bf
u}' - a {\bf 1}; {\bf v}']$ is not $\TN_p$. This shows the result for all
$X,Y$.

\item Choose $m>0$ and tuples $\bx \in \inc{X}{p}$, $\by \in \inc{Y}{p}$
such that $|m x_j|, |m y_j| < \pi/4$ for all $j$. Now,
\[
T_{W_m}[ \bx; \by ]^{\circ \alpha} = ( \cos(mx_j - my_k)^\alpha
)_{j,k=1}^p = D_\bx (1 + \tan(mx_j) \tan(my_k))^{\circ \alpha} D_\by,
\]
where $D_\bx$ for a vector $\bx$ equals ${\rm
diag}(\cos(mx_j)^\alpha)_j$. Now since $m\bx, m\by$ have increasing
coordinates, all in $(-\pi/4, \pi/4)$, Corollary~\ref{Cdet}(3) applies to
show that $T_{W_m}^\alpha$ is not $\TN_p$.

\item The previous two parts in fact show the case of $\alpha \in (0,p-2)
\setminus \Z$. The other two cases follow by using similar arguments, via
Corollary~\ref{Cdet}(1),(2). \qedhere
\end{enumerate}
\end{proof}

With Theorem~\ref{ThmA} at hand, we show that the same set of powers
(shifted by $1$) works to `preserve $\TN_p$' on Laplace transforms of
P\'olya frequency functions. Notably, the following result accommodates
\textit{all} PF functions -- in the spirit of Theorem~\ref{ThmB} -- and
not just individual ones as in Theorem~\ref{ThmA}.

\begin{cor}\label{Cthma}
Given scalars $\alpha \geq 1$ and $a_0 > 0$ and an integer $p \geq 2$,
the following are equivalent.
\begin{enumerate}
\item If $\Lambda$ is a one-sided P\'olya frequency function (i.e.~one
that vanishes on a semi-axis), then $\mathcal{B} \{ \Lambda \}^\alpha$ is
the Laplace transform of a $\TN_p$ function.

\item The exponent $\alpha \in \Z^{>0} \cup (p-1,\infty)$.
\end{enumerate}
\end{cor}

The implication $(2) \implies (1)$ was proved by Karlin in his book --
see~\cite[Chapter~7, Theorem 12.2]{Karlin}, and his (short) proof is
included in the Appendix for completeness.

\begin{proof}
In the spirit of this paper, we first reduce the test set in~(1) to a
single P\'olya frequency function, and show~(2). For the kernel
$\lambda_0(x) = e^{-x} {\bf 1}_{x > 0}$ -- which is shown to be a P\'olya
frequency function in Lemma~\ref{Llambdad} below -- standard computations
reveal:
\[
\mathcal{B} \{ \lambda_0 \}(s)^\alpha = \frac{1}{(s+1)^\alpha} =
\mathcal{B} \{ \Lambda_\alpha \}(s), \qquad \text{where }
\Lambda_\alpha(x) = \frac{e^{(\alpha-2)x}
\Omega(x)^{\alpha-1}}{\Gamma(\alpha)}.
\]
Since $\alpha \geq 1$, Laplace inversion yields: if $\mathcal{B} \{
\Lambda \}(s) = (s+1)^{-\alpha}$, then $\Lambda = \Lambda_\alpha$. Thus
$\Lambda_\alpha$ is $\TN_p$, whence so is $\Omega(x)^{\alpha-1}$ (e.g.
see the argument that concludes the proof of Lemma~\ref{Llambdad}). Now
Theorem~\ref{ThmA} shows $\alpha - 1 \in \Z^{\geq 0} \cup [p-2,\infty)$,
which implies~(2).
\end{proof}
%}}}

%{{{1 Section 3.1 - Non-degenerate extension of Karlin's result
\subsection{Non-degenerate extension of Karlin's result}\label{Shw}

We next provide the strengthening of Karlin's theorem~\ref{Tkarlin}.
Karlin was studying the function $\Omega(x) := x e^{-x} {\bf 1}_{x>0}$.
This function has Laplace transform $1 / (1+s)^2$, and is an example of a
non-smooth, one-sided P\'olya frequency function, as shown by Schoenberg
in~\cite{Schoenberg51}. More generally, given scalars $q,r > 0$, define
\[
\Omega^{(q,r)}(x) := \begin{cases}
\displaystyle \frac{qr(e^{-qx} - e^{-rx})}{r-q}, \qquad & \text{if}\ x >
0 \text{ and } q \neq r,\\
r^2 x e^{-rx}, & \text{if}\ x > 0 \text{ and } q=r,\\
0, & \text{if}\ x \leq 0.
\end{cases}
\]

By inspection, $\Omega^{(q,r)}(x) = \Omega^{(r,q)}(x) \to
\Omega^{(r,r)}(x)$ as $q \to r$. Moreover, for all $q,r > 0$ the map
$\Omega^{(q,r)}$ is a probability density function for all $q,r > 0$; and
it has bilateral Laplace transform
\[
\mathcal{B} \{ \Omega^{(q,r)} \}(s) = \frac{qr}{(s+q)(s+r)}.
\]

Thus by Schoenberg's representation theorems~\cite{Schoenberg51},
$\Omega^{(q,r)}$ is a P\'olya frequency function for all $q,r > 0$. These
functions were studied by Hirschman and Widder~\cite{HW49,HW} in greater
generality; Schoenberg~\cite{Schoenberg51} showed that they are not
smooth at the origin; and their powers are also studied in recent joint
work~\cite{BGKP-HW}. In this paper we restrict ourselves to working with
$\Omega^{(q,r)}(x)$.

Notice that $\Omega^{(q,r)}(x) = \Omega(x)$ is Karlin's kernel when the
parameters specialize to $q=r=1$. The following result extends
Theorems~\ref{ThmA} and~\ref{Tkarlin} to all other $q,r > 0$:

\begin{theorem}\label{ThmA-strong}
Fix an integer $p \geq 2$ and subsets $X,Y \subset \R$ of size at least
$p$. Also fix real numbers $q,r > 0$.
\begin{enumerate}
\item If $\alpha \geq p-2$, then $\Omega^{(q,r)}(x)^\alpha$ is $\TN_p$ on
$X \times Y$.

\item There exists $a = a(X,Y) \in \R$ such that the restriction of
$T_{\Omega^{(q,r)}_a}(x,y)^\alpha$ to $X \times Y$ (where
$\Omega^{(q,r)}_a(x) = \Omega^{(q,r)}(x-a)$), is not $\TN_p$ for all
$\alpha \in (0,p-2) \setminus \Z$. 

\item Given tuples $\bx, \by \in \inc{\R}{p}$, there exist $a \in \R$
such that the matrix
\[
(\Omega^{(q,r)}(x_j - y_k - a)^\alpha)_{j,k=1}^p
\]
is $\TP$ if $\alpha > p-2$, $\TN$ if $\alpha \in \{ 0, 1, \dots, p-2
\}$, and not $\TN$ if $\alpha \in (0,p-2) \setminus \Z$.
\end{enumerate}
\end{theorem}

Clearly, setting $q=r=1$ yields the corresponding assertions for
$\Omega(x)$ in Theorem~\ref{ThmA}. Similarly, we deduce from this result
the following extension of Corollary~\ref{C11}:

\begin{cor}\label{C33}
Fix real scalars $q,r > 0$ and $\alpha \geq 0$. The function
$\Omega^{(q,r)}(x)^\alpha$ is $\TN$ if and only if $\alpha$ is an
integer.
\end{cor}

\begin{proof}
If $\Omega^{(q,r)}(x)^\alpha$ is $\TN_p$ for all $p \geq 2$, then the
preceding theorem yields $\alpha \in \Z^{\geq 0}$. Conversely, using $0^0
:= 0$ we have that $\Omega^{(q,r)}(x)^\alpha$ is indeed $\TN$ for $\alpha
= 0,1$. Suppose $\alpha > 1$ is an integer. Clearly
$\Omega^{(q,r)}(x)^\alpha$ is integrable and non-vanishing on
$(0,\infty)$, so it suffices to show it is a $\TN$ function. This holds
for $q=r$ by the same proof as for Theorem~\ref{Tkarlin} for integer
powers (see the Appendix). If instead $q \neq r$, then we directly
compute its Laplace transform:
\[
\mathcal{B} \{ (\Omega^{(q,r)})^\alpha \}(s) =
\frac{(qr)^\alpha}{(r-q)^\alpha} \sum_{j=0}^\alpha \binom{\alpha}{j}
\frac{(-1)^j}{s + jr + (\alpha-j)q} = \frac{f(s)}{g(s)},
\]
say. Here the polynomial $g(s) := \prod_{j=0}^\alpha (s + jr +
(\alpha-j)q)$, and $f(s)$ is written by taking common denominators:
\[
f(s) := \frac{(qr)^\alpha}{(r-q)^\alpha} \sum_{j=0}^\alpha
\binom{\alpha}{j} (-1)^j \prod_{k \neq j} (s + kr + (\alpha - k) q).
\]
Clearly, $\deg(f) \leq \alpha$; but a straightforward computation shows
that at the $\alpha + 1$ points $s_0 = -(kr + (\alpha-k)q)$, $k = 0, 1,
\dots, \alpha$, we have $f(s_0) = \alpha! (qr)^\alpha$. Hence $f(s)$ is a
constant and $g(s)/f(s)$ is a polynomial, whence in the Laguerre--P\'olya
class. It follows that $\Omega^{(q,r)}(x)^\alpha$ is indeed a P\'olya
frequency function via Schoenberg's characterization of such
functions~\cite{Schoenberg51}.
\end{proof}

It remains to prove the above extension of Theorems~\ref{ThmA}
and~\ref{Tkarlin}.

\begin{proof}[Proof of Theorem~\ref{ThmA-strong}]
When $q=r$, the result follows from Theorems~\ref{ThmA} and~\ref{Tkarlin}
by a change of scale. Thus we assume henceforth -- without loss of
generality -- that $0 < q < r$.
\begin{enumerate}
\item We show the assertion for $X = Y = \R$. Suppose $\alpha \geq p-2$,
and let $\bx, \by \in \inc{\R}{p}$ for $p \geq 1$. Set $C :=
(\Omega^{(q,r)}(x_j - y_k)^\alpha)_{j,k=1}^p$, and define
\[
u_j := - e^{(q-r)x_j}, \ v_j := e^{(r-q)y_j}, \qquad 1 \leq j \leq p.
\]

Then ${\bf u}, {\bf v} \in \inc{\R}{p}$, and a straightforward
computation shows that the matrix $C$ can be described using the
Jain--Karlin--Schoenberg kernel:
\[
C = T_{\Omega^{(q,r)}} [ \bx; \by ]^{\circ \alpha} = D^{\circ \alpha}
\kjk[ {\bf u}; {\bf v} ]^{\circ \alpha} D_1^{\circ \alpha},
\]
where $D, D_1$ are diagonal matrices
\[
D = \frac{qr}{r-q} {\rm diag}( e^{-q x_1}, \dots, e^{-q x_p} ), \qquad
D_1 = {\rm diag}( e^{q y_1}, \dots, e^{q y_p} ).
\]
It follows by Theorem~\ref{ThmC} (proved below) that $C$ is $\TN$, as
desired.

\item As in the proof of Theorem~\ref{ThmA}, this part follows from the
next, by fixing $p$ elements $x_j \in X$ and $y_j \in Y$.

\item Choose $a \leq x_1 - y_p$. Using $u_j := -e^{(q-r)x_j}$ and $D_{p
\times p}$ as in the first part, and setting
\[
v'_k := e^{(r-q)(y_k+a)}, \qquad
D'_1 = {\rm diag}( e^{q (y_1+a)}, \dots, e^{q (y_p+a)} ),
\]
it follows that ${\bf u}, {\bf v}' \in \inc{\R}{p}$, that $1 + u_j v'_k >
0$ for all $j,k$, and furthermore that
\[
C' = T_{\Omega_a^{(q,r)}} [ \bx; \by ]^{\circ \alpha} = D^{\circ \alpha}
((1 + u_j v'_k)^\alpha)_{j,k=1}^p (D'_1)^{\circ \alpha}.
\]
Now the result follows from Corollary~\ref{Cdet}, akin to the proof of
Theorem~\ref{ThmA}(3).\qedhere
\end{enumerate}
\end{proof}
%}}}

%{{{1 Section 3.2 - Single-matrix encoders of Loewner convexity
\subsection{Single-matrix encoders of Loewner convexity}\label{Sconvex}

As an application of the methods used above, we provide single-matrix
encoders of the entrywise powers preserving Loewner convexity. Recall for
$I \subset \R$ that a function $f : I \to \R$ preserves \textit{Loewner
convexity} on a set $V \subset \bp_n(I)$ if $f[ \lambda A + (1-\lambda)
B] \leq \lambda f[A] + (1-\lambda) f[B]$ whenever $\lambda \in [0,1]$ and
$A \geq B \geq 0$ in $V$.

The powers preserving Loewner convexity were classified by Hiai in 2009:

\begin{theorem}[Hiai, \cite{Hiai2009}]\label{Thiai}
Let $n \geq 2$ be an integer and $\alpha \in \R$. The entrywise map
$x^\alpha$ preserves Loewner convexity on $\bp_n([0,\infty))$ if and only
if $\alpha \in \Z^{\geq 0} \cup [n, \infty)$.
\end{theorem}

In the spirit of Theorem~\ref{Tjain}, we provide single-matrix encoders
of these powers:

\begin{theorem}\label{Tconvex}
Let $n \geq 2$ be an integer and $\alpha \in \R$. Suppose $x_1, \dots,
x_n \in \R$ are pairwise distinct, non-zero scalars such that $1 + x_j
x_k > 0$ for all $j,k$. Let $A := (1 + x_j x_k)_{j,k=1}^n$ and $B := {\bf
1}_{n \times n}$. Then $x^\alpha$ preserves Loewner convexity on $A \geq
B \geq 0$ if and only if $\alpha \in \Z^{\geq 0} \cup [n, \infty)$.
\end{theorem}

The proof relies on the following preliminary lemma, which can be shown
by an argument of Hiai -- see the Appendix.

\begin{lemma}\label{Lconvex}
Let $n \geq 2$ and $A \geq B \geq 0$ in $\bp_n(\R)$ be such that $A - B =
u u^T$ has rank one and no non-zero entries. Choose an open interval $I
\subset \R$ containing the entries of $A,B$, and suppose $f : I \to \R$
is differentiable. If the entrywise map $f[-]$ preserves Loewner
convexity on the interval
\[
[B,A] := \{ \lambda A + (1-\lambda) B : \lambda \in [0,1] \}
\]
then $f'[-]$ preserves Loewner monotonicity on $[B,A]$. The converse
holds for arbitrary matrices $0 \leq B \leq A$.
\end{lemma}

We now prove Theorem~\ref{Tconvex} -- in the process also proving Hiai's
result:

\begin{proof}[Proof of Theorems~\ref{Tconvex} and~\ref{Thiai}]
By Lemma~\ref{Lconvex} and Theorem~\ref{Tfitzhorn}(2), $x^\alpha$
preserves Loewner convexity on $\bp_n((0,\infty))$ for $\alpha \in
\Z^{>0} \cup [n,\infty)$, and obviously so for $\alpha = 0$. The result
for $\bp_n([0,\infty))$ follows by continuity. Next, if $x^\alpha$
preserves Loewner convexity on $\bp_n([0,\infty))$, then it does so on
the given matrices $A \geq B \geq 0$. Finally, if the latter condition
holds and $\alpha \not\in \Z^{\geq 0}$, then Lemma~\ref{Lconvex} applies,
so $\alpha \geq n$ via Theorem~\ref{Tjain}(2).
\end{proof}
%}}}

%{{{1 Section 4 - Hankel $\TN_p$ kernels: preservers, critical exponent,
%and Theorem~\ref{ThmB}
\section{Hankel $\TN_p$ kernels: preservers, critical exponent, and
Theorem~\ref{ThmB}}

In this section we first prove Theorem~\ref{ThmB}. The key tool is a
result of Fekete from 1912~\cite{Fe}, subsequently strengthened by
Schoenberg in 1955~\cite{Schoenberg55}:

\begin{lemma}\label{Lfekete}
Suppose $1 \leq p \leq m,n$ are integers, and $A \in \R^{m \times n}$.
Then $A$ is $\TP_p$ if and only if all contiguous minors of orders $\leq
p$ are positive. (Here, `contiguous' means that the rows and columns for
the minor are both consecutive.)
\end{lemma}

The proof is not too long, relying on computational lemmas by Gantmacher
and Krein. See~\cite{GK}.

\begin{cor}\label{Cfekete}
Suppose $1 \leq p \leq n$ are integers and $A \in \R^{n \times n}$ is
Hankel. Let $A^{(1)}$ denote the truncation of $A$, i.e.~the submatrix
with the first row and last column of $A$ removed. Then $A$ is $\TN_p$
(respectively $\TP_p$) if and only if every contiguous principal minor of
$A$ and of $A^{(1)}$ of size $\leq p$ is non-negative (respectively
positive).
\end{cor}

This result can be found in~\cite[Chapter 4]{Pinkus} for the $\TP$ case,
and in~\cite{FJS} for the $\TP_p, \TN, \TN_p$ cases. These sources do not
use the word `contiguous' -- the advantage of using contiguous
(principal) minors is that they are all Hankel. We provide a quick proof
of Corollary~\ref{Cfekete} in the Appendix, for self-completeness. For
now, we apply this result to prove Theorem~\ref{ThmB} and other results.
The relevant part of this argument is isolated into the following
standalone result.

\begin{prop}\label{Phankeltnp}
Suppose $p \geq 2$ is an integer, $X \subset \R$ is an interval with
positive measure, and $H : X \times X \to \R$ is a continuous Hankel
$\TN_p$ kernel. If $f : [0,\infty) \to [0,\infty)$ is continuous at
$0^+$, and preserves positive semidefiniteness when acting entrywise on
$r \times r$ Hankel matrices for $1 \leq r \leq p$, then $f \circ H : X
\times X \to \R$ is continuous, Hankel, and $\TN_p$.
\end{prop}

\begin{proof}
The first step is to show that $f$ is continuous on $(0,\infty)$; this
quickly follows e.g.~from work of Hiai~\cite{Hiai2009}, and is sketched
in the Appendix for completeness. (A longer proof is via using a 1929
result of Ostrowski; see e.g.~\cite{BGKP-hankel}.) Thus $f \circ H$ is
continuous and Hankel on $X \times X$.

Now let $2 \leq r \leq p$ and choose $\bx, \by \in \inc{X}{r}$. We need
to show $\det (f \circ H)[ \bx; \by ] \geq 0$. Let ${\bf u} = (u_1,
\dots, u_m)$ denote the ordered tuple whose coordinates are the union of
the $x_j, y_k$ (without repetitions). We claim that $(f \circ H)[ {\bf
u}; {\bf u} ]$ is $\TN_p$; this would suffice to complete the proof.

To show the claim, approximate the increasing tuple ${\bf u}$ by tuples
${\bf u}^{(k)} \in \inc{(X \cap \mathbb{Q})}{m}$ of rational numbers in
$X$, with ${\bf u}^{(k)} \to {\bf u}$ as $k \to \infty$. Choose integers
$N_k > 0$ such that $N_k {\bf u}^{(k)}$ has integer coordinates. Now if
the matrices
\[
(f \circ H)[ {\bf v}_k; {\bf v}_k ], \quad \text{where} \quad {\bf
v}_k := (u^{(k)}_1, u^{(k)}_1 + \frac{1}{N_k}, \dots, u^{(k)}_m)
\]
can be shown to be $\TN_p$, then by taking submatrices and the limit as
$k \to \infty$, it follows that $(f \circ H)[ {\bf u}; {\bf u} ]$ is
$\TN_p$, as claimed. We use here that $f$ is continuous.

Since each ${\bf v}_k$ is an arithmetic progression, it is easy to see
that the matrices $A_k := H[ {\bf v}_k; {\bf v}_k ]$ are Hankel, and
$\TN_p$ because $H$ is so. Now observe that all contiguous principal
submatrices $C$ of $A_k$ or of $A_k^{(1)}$ of size $2 \leq r \leq p$ are
symmetric Hankel positive semidefinite matrices. Thus $f[C]$ is positive
semidefinite by assumption, hence has determinant $\geq 0$. It follows by
Corollary~\ref{Cfekete} that $(f \circ H)[ {\bf v}_k; {\bf v}_k ]$ is
$\TN_p$ for all $k$, and this completes the proof.
\end{proof}

With Proposition~\ref{Phankeltnp} and the previous results at hand, our
next main result follows.

\begin{proof}[Proof of Theorem~\ref{ThmB}]
The first step is to verify that $H_{u_0}$ is $\TN$; this is easy because
$H_{u_0}$ has `rank two', being the moment sequence/kernel of the
two-point measure $\delta_1 + c_0 \delta_{u_0}$, so all $r \times r$
minors vanish for $r \geq 3$.
We next prove a chain of cyclic implications. Clearly $(1) \implies (2)$,
and $(3) \implies (1)$ by Proposition~\ref{Phankeltnp} and
Theorem~\ref{Tfitzhorn}(1). Finally, suppose $\alpha \not\in \{ 0, 1,
\dots, p-2 \}$. Choose tuples $\bx, \by \in \inc{X_0}{p}$ and apply
Corollary~\ref{Cdet} with $x_j, y_j$ replaced by $\sqrt{c_0} u_0^{x_j},
\sqrt{c_0} u_0^{y_j}$ respectively; we also reverse the rows and columns
if $u_0 \in (0,1)$. This yields a $\TN$ matrix $H_{u_0}[ \bx; \by ]$,
whose $\alpha$th entrywise power is not $\TN_p$ if $\alpha \in (0,p-2)
\setminus \Z$, and is $\TP$ if $\alpha > p-2$. This shows both $(2)
\implies (3)$ as well as the remaining assertions.
\end{proof}

For the curious reader, Theorem~\ref{ThmB} leads to a question about
Toeplitz analogues that may be of theoretical interest. One can ask if
this `clean' phenomenon holds for the parallel class of Toeplitz kernels
-- namely, if for all integers $p \geq 2$, the $\TN_p$-preserving powers
$x^\alpha$ are precisely $\alpha \in \Z^{\geq 0} \cup [p-2,\infty)$. This
is easily verified to hold for $p=2$; see e.g.~\cite{Schoenberg51} (or
Lemma~\ref{Ltn2} below), where $\TN_2$ functions are characterized as
exponentials of concave functions. However, such a clean result fails to
hold in general. Specifically, considering the question from the `dual'
viewpoint of the powers $\alpha$: while $x^\alpha$ for $\alpha = 0,1$
obviously preserves $\TN_p$ for all $p$, this fails to hold for every
other integer $\alpha \geq 2$.
Namely, one can find a $\TN_p$ kernel (for some $p \geq 0$), whose
$\alpha$th power is not $\TN_p$. This can be refined further, to work
with a single kernel -- which is moreover $\TN$ -- that provides a
counterexample for all integer powers:

\begin{lemma}\label{Lweinberger}
There exists a P\'olya frequency function $M : \R \to \R$, such that for
every integer power $\alpha \geq 2$, there exists an integer $p(\alpha)
\geq 1$ satisfying: $M^\alpha$ is not $\TN_{p(\alpha)}$.
\end{lemma}

\begin{proof}
Let $M(x) := 2e^{-|x|} - e^{-2|x|}$ for $x \in \R$. It was shown
in~\cite{BGKP-TN} that $M^\alpha$ is not $\TN$ for any $\alpha \geq 2$,
while $M$ is. (See the Appendix for details.) This proves the result.
\end{proof}

In light of Lemma~\ref{Lweinberger}, one can ask more refined questions,
e.g.~if all non-integer powers $\alpha > p-2$ preserve $\TN_p$ Toeplitz
functions/kernels, with $p \geq 4$. A challenge in tackling such
questions comes from the lack of a well-developed theory for P\'olya
frequency functions of finite order, i.e., integrable $\TN_p$ functions.
For instance, to our knowledge there was no known characterization to
date of P\'olya frequency functions of order $p = 4, 5, \dots$. (While
Theorem~\ref{ThmE} now fulfils this need, it does not provide enough
information to help here.)

\begin{remark}
In light of Lemma~\ref{Lweinberger} and the above results, one can also
ask about the classification of powers -- or more generally, arbitrary
functions -- that preserve the class of $\TN$ kernels, whether Hankel or
Toeplitz, upon composing. These characterizations were recently achieved
in joint work~\cite{BGKP-TN}: for continuous Hankel kernels, the
preservers are precisely the convergent power series with non-negative
Maclaurin coefficients (see also Lemma~\ref{Lps}), while for Toeplitz
kernels, the preservers are precisely constants $c$ or homotheties $cx$
or Heaviside functions $c {\bf 1}_{x>0}$, with $c \geq 0$.
\end{remark}
%}}}

%{{{1 Section 4.1 - Connection to fixed-dimension results on positivity preservers
\subsection{Connection to fixed-dimension results on positivity
preservers}\label{Sfixeddim}

Given an integer $p \geq 1$ and a subset $I \subset \R$, let $\bp_p(I)$
denote the set of real symmetric $p \times p$ matrices, which are
positive semidefinite and have all entries in $I$. The critical exponent
phenomena studied above suggest that $\TN_p$-preservers are closely
related to entrywise functions preserving positive semidefiniteness on
$\bp_p((0,\infty))$ -- especially for Hankel kernels, in light of
Proposition~\ref{Phankeltnp}. Although our focus in this paper is on
powers, we briefly digress to point out a few such connections. The first
is Loewner's necessary condition for preserving positivity on such
matrices:

\begin{theorem}[Loewner / Horn, 1969, \cite{horn}]\label{Thorn}
Suppose $I = (0,\infty)$, $f : I \to \R$ is continuous, and $p \geq 3$ is
an integer such that $f[-]$ applied entrywise to matrices in $\bp_p(I)$
preserves positivity. Then $f \in C^{p-3}(I)$, $f^{(p-3)}$ is convex on
$I$, and $f, f', \dots, f^{(p-3)} \geq 0$ on $I$. If in particular $f \in
C^{p-1}(I)$, then $f^{(p-2)}, f^{(p-1)} \geq 0$ on $I$ as well.
\end{theorem}

We claim that the same conclusions hold if $f$ preserves the $\TN_p$
Hankel kernels -- in fact on a far smaller test set, and without the
continuity assumption from~\cite{horn}:

\begin{theorem}
Suppose $I = (0,\infty)$, $f : I \to \R$, and $X_0 \subset \R$ is any
interval with positive measure. Suppose $p \geq 3$ is an integer such
that the post-composition transform $f \circ -$ preserves $\TN_p$ on
Hankel $\TN$ kernels corresponding to non-negative measures supported on
at most two points. Then the conclusions of Theorem~\ref{Thorn} hold.
\end{theorem}

That this result is sharp -- in the number of non-negative derivatives
$f, \dots, f^{(p-1)}$ on $I$ -- follows from Theorem~\ref{ThmB}, by
considering a suitable power function $f$.

\begin{proof}
We appeal to results in~\cite{BGKP-hankel}, which assert that if $f[-]$
preserves positivity on the matrices
\begin{align*}
&\ (a_0 + c_0 u_0^{j+k})_{j,k=0}^{p-1}, \qquad \ \ a_0, c_0 \geq 0,\ a_0
+ c_0 > 0,\\
& \begin{pmatrix} a & b \\ b & b \end{pmatrix}, \ \begin{pmatrix} c^2 &
cd \\ cd & d^2 \end{pmatrix}, \qquad a,b,c,d > 0, \ a > b > 0,\ c \geq d > 0
\end{align*}
for some fixed $u_0 \in (0,1)$, then $f$ satisfies the conclusions of
Theorem~\ref{Thorn}.
It thus suffices to embed these test matrices in $\TN$ Hankel kernels. We
do so on $\R \times \R$; the restriction to $X_0 \times X_0$ follows by a
linear change of variables that contains an appropriate compact
sub-interval of $\R$.
The first class of test matrices above embeds in the Hankel kernels
\[
H_{a_0, c_0}(x,y) := a_0 + c_0 u_0^{x+y}, \qquad x,y \in \R,
\]
for $a_0,c_0 \geq 0$, while the `rank-one' matrices above embed in the
kernel $H_{c,0}$ if $c=d$, and in $H_{0,c^2}$ with $u_0 = d/c$, if
$c>d>0$. Recently in~\cite{BGKP-TN}, the remaining class of matrices
$\begin{pmatrix} a & b \\ b & b\end{pmatrix}$ above was shown to embed in
the following `rank-two' $\TN$ Hankel kernel, which completes the proof:
\[
\frac{( 2 a - b )^2}{4 a - 3 b} %
\left( \frac{b}{2 a - b}\right)^{ x + y } + %
\frac{b ( a - b )}{4 a - 3 b} 2^{ x + y }, \qquad x,y \in \R. \qedhere
\]
\end{proof}

The next connection is to an -- even older -- observation of P\'olya and
Szeg\H{o}~\cite{polya-szego} from 1925:

\begin{lemma}\label{Lps}
Suppose $f_0$ is the restriction to $[0,\infty)$ of an entire function
with non-negative Maclaurin coefficients. Then $f_0 \circ -$ preserves
the class of continuous $\TN_p$ Hankel kernels on $X \times X$, for all
integers $p \geq 1$ and intervals $X \subset \R$.
\end{lemma}

\begin{proof}
By the Schur product theorem, $x^k$ entrywise preserves positivity on
$\bp_p([0,\infty))$ for all integers $k \geq 0$; here we set $0^0 := 1$.
Since $\bp_p([0,\infty))$ is a closed convex cone, it follows that all
functions $f_0$ as in the lemma share the same property. We are now done
by Proposition~\ref{Phankeltnp}.
\end{proof}

Our third connection is to entrywise polynomials that preserve $\TN_p$.
By the preceding lemma, all power series with non-negative coefficients
preserve $\TN_p$ on continuous Hankel $\TN_p$ kernels. It is natural to
ask is if a wider class of polynomials shares this property.\footnote{In
the original setting of entrywise polynomials and power series preserving
positivity on $\bp_p((0,\infty))$, no examples were known for $p \geq 3$,
until recent joint work~\cite{KT}.}
We conclude this section by providing a positive answer, essentially
coming from recent joint work with Tao~\cite{KT}:

\begin{theorem}\label{Tunbdd}
Let $p > 0$ and $0 \leqslant n_0 < \dots < n_{p-1} < M < n_p < \dots <
n_{2p-1}$ be integers, and let $c_{n_0},\dots,c_{n_{2p-1}} > 0$ be reals.
There exists a negative number $c_M$ such that the polynomial
\[
 x \mapsto c_{n_0} x^{n_0} + c_{n_1} x^{n_1} + \dots + c_{n_{p-1}}
 x^{n_{p-1}} + c_M x^M + c_{n_p} x^{n_p} + \dots + c_{n_{2p-1}}
 x^{n_{2p-1}},
\]
preserves the continuous Hankel $\TN_p$ kernels on $X \times X$, for
intervals $X \subset \R$ with positive measure.
\end{theorem}

Via Proposition~\ref{Phankeltnp}, Theorem~\ref{Tunbdd} follows
from~\cite{KT}, because such a polynomial was shown in \textit{loc.~cit.}
to preserve Loewner positivity on $\bp_p([0,\infty))$.
Theorem~\ref{Tunbdd} also admits extensions to power series and more
general preservers; we refer the interested reader to~\cite{KT} for
further details.
%}}}

%{{{1 Section 5 - Theorem~\ref{ThmC}: Critical exponent for the
%Jain--Karlin--Schoenberg kernel
\section{Theorem~\ref{ThmC}: Critical exponent for total positivity of
the Jain--Karlin--Schoenberg kernel}

We next show Theorem~\ref{ThmC} on the total non-negativity of the powers
of the kernel $\kjk$, and explain how it connects to the (total)
positivity results stated before it in the opening section.

\begin{proof}[Proof of Theorem~\ref{ThmC}]
The second part follows from Corollary~\ref{Cdet}. 
For the first, begin with the basic trigonometric fact:
\textit{If $-\pi/2 < \varphi < \theta < \pi/2$, then $\tan (\theta) \tan
(\varphi) > -1$ if and only if $\theta - \varphi < \pi/2$.}

Now let $\bx, \by \in \inc{\R}{p}$ and let $u_j := \tan^{-1}(x_j), \ v_j
:= \tan^{-1}(y_j)$. Then ${\bf u}, {\bf v} \in \inc{(-\pi/2, \pi/2)}{p}$,
so:
\begin{align*}
\kjk(x_j, y_k) = &\ (1 + \tan(u_j) \tan(v_k)) {\bf 1}_{\tan(u_j)
\tan(v_k) > -1}\\
= &\ (1 + \tan(u_j) \tan(v_k)) {\bf 1}_{|u_j - v_k| < \pi/2}\\
= &\ \sec(u_j) \sec(v_k) \left[ \cos(u_j - v_k) {\bf 1}_{|u_j - v_k| <
\pi/2} \right]\\
= &\ \sec(u_j) \sec (v_k) T_W(u_j, v_k).
\end{align*}
It follows that
\begin{equation}\label{Earctan}
\kjk[ \bx; \by ]^{\circ \alpha} = D_{\bf u}^\alpha T_W[ {\bf u}; {\bf v}
]^{\circ \alpha} D_{\bf v}^\alpha, \qquad \forall \alpha \geq 0
\end{equation}
where $D_{\bf u}$ for a vector ${\bf u} \in \inc{(-\pi/2,\pi/2)}{p}$ is
the diagonal matrix with $(j,j)$ entry $\sec(u_j)$.
Theorem~\ref{Tschoenberg} now implies that this matrix is $\TN$ if
$\alpha \geq p-2$, proving the first part.

Finally, we show the third part. Since the kernel $\kjk$ is
invariant under the automorphism group generated by the involutions $x
\leftrightarrow y$ and $(x,y) \leftrightarrow (-x,-y)$, it suffices to
show that the restriction to $[0, \infty) \times \R$ of $\kjk^\alpha$ is
$\TN_p$ if and only if $\alpha \in \Z^{\geq 0} \cup [p-2,\infty)$. This
already holds for $\alpha \geq p-2$ from above; and it does not hold for
$\alpha \in (0,p-2) \setminus \Z$ by assertion~(2)(c) shown above. The
final sub-case is when $\alpha \in \Z^{\geq 0}$. Let $\bx \in
\inc{[0,\infty)}{p}$ and $\by \in \inc{\R}{p}$; we need to show that
\[
\det C^{\circ \alpha} \geq 0, \quad \text{where} \quad C := (\max(1 + x_j
y_k,0))_{j,k=1}^p.
\]
By the continuity of the function $\kjk$, we may assume $x_1 > 0$. Now,
\[
C^{\circ \alpha} = {\rm diag}(x_j^\alpha) (\max(0, x_j^{-1} -
(-y_k))^\alpha)_{j,k=1}^p = {\rm diag}(x_j^\alpha e^{\alpha x_j^{-1}})
(\Omega(x_j^{-1} - (-y_k)))^{\circ \alpha} {\rm diag}(e^{\alpha y_k}).
\]
Reversing the rows and the columns, we are done by
Theorem~\ref{Tkarlin}.\footnote{This part of Karlin's result, for integer
powers $\alpha \geq 0$, was already shown by Schoenberg
in~\cite{Schoenberg51}. For the interested reader, his direct proof is
included in the Appendix.}
\end{proof}

\begin{remark}\label{Rmaster}
We now explain how Theorem~\ref{ThmC} implies many of the results in
Section~\ref{S1}.
\begin{enumerate}
\item Given scalars
$0 <x_1 < \cdots < x_p$ and $y_1 < \cdots < y_p$,
the Karlin-kernel $\Omega$ is a specialization of the
Jain--Karlin--Schoenberg kernel, up to multiplying by diagonal matrices
and reversing rows and columns:
\begin{equation}\label{ECimpliesA}
(T_\Omega [ \bx; \by ]^{\circ \alpha})^T = D^{\circ \alpha} \kjk[ \by';
\bx' ]^{\circ \alpha} D_1^{\circ \alpha},
\end{equation}
where $\by' = (-y_1, \dots, -y_p)$, $\bx' = (1/x_1, \dots, 1/x_p)$, and
$D_1, D$ are diagonal matrices
\[
D_1 = {\rm diag}( x_p e^{-x_p}, \dots, x_1 e^{-x_1} ), \qquad
D = {\rm diag}( e^{y_p}, \dots, e^{y_1} ),
\]

\item The proof of Theorem~\ref{ThmA-strong} has a similar computation
as~\eqref{ECimpliesA}, with $\Omega$ replaced by $\Omega^{(q,r)}$.

\item Similarly, the proof of Theorem~\ref{ThmC}(1) shows how, via the
transformation $\arctan$, the Jain--Karlin--Schoenberg kernel is
intimately related to the Schoenberg-kernel $T_W$. These observations
show how Theorem~\ref{ThmC} about (the powers of) the
Jain--Karlin--Schoenberg kernel is related to Theorem~\ref{ThmA}, and to
Theorems~\ref{Tschoenberg} and~\ref{Tkarlin} of Schoenberg and Karlin,
respectively.

\item Given an integer $n \geq 2$, the kernel $\kjk$ clearly specializes
on the set of bi-tuples
\[
\{ (\bx, \by) \in (\inc{\R}{n})^2 : 1 + x_j y_k > 0\ \forall j,k = 1,
\dots, n \}
\]
to Jain's theorem~\ref{Tjain}(1) -- in fact, to the stronger $\TN$
assertion in Theorem~\ref{ThmC}(2).

\item Restricting the kernel $\kjk$ to $(0,\infty)^2$ via the transform
$u_0^x$, we see that Theorem~\ref{ThmC} implies the equivalence $(2)
\Longleftrightarrow (3)$ in Theorem~\ref{ThmB}.

\item Our methods have provided an alternate proof above to Karlin's
theorem~\ref{Tkarlin}. Indeed, as discussed during the proof of
Theorem~\ref{ThmC}(3), the result is shown in the Appendix for integer
powers, and for non-integer powers $\alpha > p-2$ it is a special case of
Schoenberg's theorem -- transforming the domain from $(-\pi/2, \pi/2)^2$
to $\R^2$ via $\arctan$, then restricting to $\R \times [0,\infty)$. Here
we use the identifications of $\kjk$ with Schoenberg and Karlin's
kernels.
\end{enumerate}
\end{remark}

In fact, it is possible to refine the above results even more. Given
integers $1 \leq p \leq n$, matrices $C = (1+x_jy_k)_{j,k=1}^n$ with
positive entries, and powers $\alpha \geq 0$, one can show that all $p
\times p$ minors of $C^{\circ \alpha}$ have the same sign -- which
depends only on $n,p,\alpha$ but not on $x_j,y_k$. This follows from
above for $\alpha \in \Z^{\geq 0} \cup [p-2,\infty)$. If $\alpha \in
(0,p-2) \setminus \Z$, this follows by using SSR (strictly sign regular)
matrices and kernels, found in Karlin's book~\cite{Karlin} and Jain's
works~\cite{Jain,Jain2}. In fact, the following holds, e.g.~by
Propositions~\ref{Phomotopy2} and~\ref{Pjain2}, and \cite[Theorem
2.4]{Jain}:

\begin{cor}\label{Cssr}
Given a scalar $\alpha \geq 0$, an integer $n \geq 2$, and tuples $\bx,
\by \in \inc{\R}{n}$ such that $1 + x_j y_k > 0$ for all $j,k$, the
power-matrix $C^{\circ \alpha}$ studied above is sign regular, with
signature given as follows:
\[
{\rm signature} ((1 + x_j y_k)^\alpha)_{j,k=1}^n =
\begin{cases}
((-1)^{\lfloor p/2 \rfloor} \varepsilon_{p,\alpha})_{p=1}^n, & \text{if}\
\alpha \not\in \{ 0, 1, \dots, n-2 \},\\
(((-1)^{\lfloor p/2 \rfloor} \varepsilon_{p,\alpha})_{p=1}^{\alpha+1},
0,\dots,0), \quad & \text{otherwise.}
\end{cases}
\]
That is, the sign of any $p \times p$ minor of
$((1+x_jy_k)^\alpha)_{j,k=1}^n$ depends only on $n,p,\alpha$; here,
$\varepsilon_{p,\alpha}$ equals
\[
\varepsilon_{p,\alpha} = \begin{cases}
(-1)^{\lfloor p/2 \rfloor}, & \text{if}\ \alpha > p-2,\\
(-1)^{p-s+1}, \qquad & \text{if}\ 2s < \alpha < 2s+1 \leq p-2, \ s \in
\Z^{\geq 0},\\
(-1)^{s+1}, & \text{if}\ 2s+1 < \alpha < 2s+2 \leq p-2, \ s \in \Z^{\geq
0},\\
0, & \text{if}\ \alpha = 0, 1, \dots, p-2.
\end{cases}
\]
\end{cor}

To conclude this section, note that Theorem~\ref{ThmC} completely
classifies the powers of $\kjk$ preserving $\TN_p$ on $\R \times [0,
\infty)$. The same question on the full domain $\R^2$ of $\kjk$ remains,
but only for integers $\alpha \in \{ 0, 1, \dots, p-2 \}$. This is
equivalent to the following

\begin{question}\label{Qconj}
For an integer $\alpha \geq 0$, can the kernel $\kjk^\alpha$ be shown to
not be $\TN_{\alpha+3}$ on $\R \times \R$? More strongly, can it be shown
to not be `positive semidefinite', i.e.~using $\bx = \by \in
\inc{\R}{\alpha+3}$?
\end{question}

A complete resolution of Question~\ref{Qconj} would complete the
classification of powers of the Jain--Karlin--Schoenberg kernel $\kjk$
that are totally non-negative of each order $p \geq 2$. (It would also
complete the remaining sub-case of integer $\alpha \leq p$ in
Theorem~\ref{ThmA-strong}.) Note that this question for $\kjk$ has a
`positive' answer for $\alpha = 0,1$, so that $\kjk$ is not $\TN_4$.
Indeed,
\begin{alignat*}{3}
& \bx = \by = (-1,0,1) \in \inc{\R}{3} & \quad \implies & \quad \det
\kjk[ \bx; \by ]^{\circ 0} = \det \begin{pmatrix} 1 & 1 & 0 \\ 1 & 1 & 1
\\ 0 & 1 & 1 \end{pmatrix} = -2,\\
& \bx = \by = \frac{1}{\sqrt{2}}(-2,-1,1,2) \in \inc{\R}{4} & \quad
\implies & \quad \det \kjk[ \bx; \by ] = \det \begin{pmatrix}
3 & 2 & 0 & 0 \\
2 & 3/2 & 1/2 & 0\\
0 & 1/2 & 3/2 & 2\\
0 & 0 & 2 & 3
\end{pmatrix} = -2,
\end{alignat*}
where the $\alpha = 0$ case uses $0^0 = 0$.
%Finally, it is natural to ask if Corollary~\ref{Cssr} extends to all of
%$\kjk$, allowing for zero minors:
%
%\begin{question}
%For real $\alpha \geq 0$, is the kernel $\kjk^\alpha$ sign-regular on
%5$\R^2$? What is its signature?
%\end{question}
%}}}

%{{{1 Section 6 - Theorem~\ref{ThmD}: Laplace transform of compactly
%supported $\TN_p$ function
\section{Theorem~\ref{ThmD}: Laplace transform of a compactly supported
$\TN_p$ function}

We now show Theorem~\ref{ThmD}. The first step toward proving the result
is to characterize $\TN_2$ functions $\Lambda$ on a sub-interval $I
\subset \R$, instead of on all of $\R$ as is prevalent in the literature.
We provide a proof of this result along the lines of~\cite{Schoenberg55},
but with a few modifications for more general $I$:

\begin{lemma}\label{Ltn2}
Suppose $J \subset \R$ is an interval strictly containing the origin, and
$\Lambda : J - J \to \R$ is Lebesgue measurable. The following are
equivalent:
\begin{enumerate}
\item The nonzero-locus of $\Lambda$ is an interval $I \subset J - J$, on
which $\Lambda > 0$ and $\log \Lambda$ is concave.

\item The Toeplitz kernel $T_\Lambda : J \times J \to \R$ is $\TN_2$.
\end{enumerate}
Thus $\Lambda$ is continuous on the interior of $I$, whence discontinuous
on $J - J$ at most at two points.
\end{lemma}

In particular, this applies to $I = (-\rho/2, \rho/2) \subset J =
(-\widetilde{\rho}/2, \widetilde{\rho}/2)$, as in Theorem~\ref{ThmD}.

\begin{proof}
The result is straightforward if $\Lambda$ does not vanish at most at one
point, so we suppose henceforth that $\Lambda \neq 0$ at least at two
points.

$(1) \implies (2)$: Given scalars $\alpha < \beta$ and $\gamma < \delta$
in $J$, note that $\alpha - \gamma, \beta - \delta \in (\alpha - \delta,
\beta - \gamma)$. If $\alpha - \gamma$ or $\beta - \delta$ lie outside
$I$, the matrix $M := \begin{pmatrix} \Lambda(\alpha - \gamma) &
\Lambda(\alpha - \delta) \\ \Lambda(\beta - \gamma) & \Lambda(\beta -
\delta) \end{pmatrix}$ has a zero row or zero column. Else $\alpha -
\gamma, \beta - \delta \in I$; if now one of $\alpha - \delta, \beta -
\gamma$ is not in $I$ then $M$ is triangular, whence again $\det(M) \geq
0$. Else $M$ has all positive entries; now the concavity of $\log
\Lambda$ implies $\det(M) \geq 0$.

$(2) \implies (1)$:
Since $\Lambda$ is $\TN_2$, we have $\Lambda \geq 0$ on $J - J$. Fix
$\delta > 0$ such that $J$ contains either $[0,\delta)$ or $(-\delta,
0]$. Suppose $\Lambda(x_0) > 0$. We claim that if $x_1 > x_0$ in $J - J$
and $\Lambda(x_1) = 0$, then $\Lambda$ vanishes on $(J - J) \cap [x_1,
\infty)$; and similarly for $x_1 < x_0$ in $J - J$. It suffices to show
that $\Lambda(y) = 0$ for $y \in (J - J) \cap (x_1, x_1 + \delta)$. If $J
\supset (-\delta,0]$, this is because
\[
0 \leq \det T_\Lambda[ (x_0, x_1); (x_1-y, 0)] = \det \begin{pmatrix}
\Lambda(x_0 - x_1 + y) & \Lambda(x_0) \\ \Lambda(y) & \Lambda(x_1)
\end{pmatrix} = -\Lambda(x_0) \Lambda(y);
\]
here, $x_0 - x_1 + y \in (x_0,y) \subset J - J$. Similarly, if $J \supset
[0, \delta)$, then we instead use
\[
0 \leq \det T_\Lambda[ (x_0 - x_1 + y, y); (0, y -x_1)] = \det
\begin{pmatrix} \Lambda(x_0 - x_1 + y) & \Lambda(x_0) \\ \Lambda(y) &
\Lambda(x_1) \end{pmatrix} = -\Lambda(x_0) \Lambda(y).
\]

This produces the interval $I$; now given points $y - \epsilon < y < y +
\epsilon$ of $I$, we show that $\Lambda(y) \geq \sqrt{\Lambda(y +
\epsilon) \Lambda(y - \epsilon)}$ using discrete-time, finite state-space
Markov chains. Let $n_0 := 2 \lceil \epsilon / \delta \rceil$, so that
$\epsilon / n_0 \in (0, \delta)$. Let $z_k := \Lambda(y + k
\epsilon/n_0)$ for $-n_0 \leq k \leq n_0$; then $z_k > 0$. Now if
$J \supset (-\delta,0]$, then
\[
0 \leq \det T_\Lambda[ (y-(k+1)\epsilon/n_0, y - k \epsilon/n_0);
(-\epsilon/n_0, 0)] = z_k^2 - z_{k-1}z_{k+1}, \qquad \forall -n_0 < k <
n_0.
\]
If instead $J \supset [0,\delta)$ then we use
$0 \leq \det T_\Lambda[ (y-k\epsilon/n_0, y - (k-1) \epsilon/n_0);
(0, \epsilon/n_0)]$ for the same values of $k$, to obtain the same
conclusions.
From each case, it follows inductively that
\[
z_0 \geq (z_1 z_{-1})^{1/2} \geq (z_2 z_0^2 z_{-2})^{1/4} \geq \cdots
\geq \prod_{j=0}^{n_0} z_{2j-n_0}^{\binom{n_0}{j} / 2^{n_0}} \geq \cdots
\]

At each step, no power of $z_{\pm n_0}$ is changed, while the remaining
powers $z_j^\gamma$ are lower-bounded by $(z_{j-1} z_{j+1})^{\gamma/2}$.
The exponents of the $z_j$ give probability distributions on $\mathcal{S}
:= \{ -n_0, \dots, 0, 1, \dots, n_0 \}$ corresponding to the symmetric
gambler's ruin, i.e.~a simple random walk on the state space
$\mathcal{S}$ with absorbing barriers $z_{\pm n_0}$. The transition
probabilities here for all other states $z_j$ are $1/2$ for $z_j \mapsto
z_{j \pm 1}$. Since at each stage we moreover have equal powers of
$z_{\pm n_0}$, it follows by Markov chain theory (or one can show via a
direct argument)\footnote{Indeed, if $c_t$ denotes the sum of the
exponents for $z_{-(n_0-1)}, \dots, z_0, z_1, \dots, z_{n_0-1}$ at `time'
$t$, then one shows via the AM--GM inequality that $c_{t+(2n_0-1)} \leq
c_t (1 - 2^{1-n_0})$. Now let $t = m(2n_0-1)$, with $m \to \infty$.} that
$z_0 \geq \sqrt{z_{n_0} z_{-n_0}}$. Hence $-\log \Lambda$ is
midpoint-convex and measurable on $I$. It follows by Sierpi\'nsky's
well-known result~\cite{Sierpinsky} that $-\log \Lambda$ is continuous on
the interior of $I$, whence convex, and so $\Lambda$ is also continuous
on the interior of $I$.

Finally, to show $-\log \Lambda$ is convex on $I$, it suffices to show
for $a,b \in I$ and $\lambda \in (0,1)$ that $\log \Lambda(\lambda a +
(1-\lambda)b) \geq \lambda \Lambda(a) + (1-\lambda) \Lambda(b)$. But this
can be shown by approximating $\lambda$ by dyadic rationals $\lambda_n
\in (0,1)$ for all $n \geq 1$. For each of these, the above mid-convexity
implies:
\[
\log \Lambda(\lambda_n a + (1-\lambda_n)b) \geq \lambda_n \log \Lambda(a)
+ (1-\lambda_n) \log \Lambda(b), \qquad \forall n \geq 1.
\]
Letting $n \to \infty$, since $\Lambda$ is continuous on the interior of
$I$, it follows that $-\log \Lambda$ is convex on $I$. This completes
the proof of $(2) \implies (1)$.
\end{proof}

\begin{proof}[Proof of Theorem~\ref{ThmD}]
%We modify as follows, the arguments of Schoenberg's proof of Theorem~$4$
%in~\cite{Schoenberg55}.
Let $0 < \epsilon < \rho/2 \leq \widetilde{\rho} - \rho/2$, and work with
integers $m > (p-1)\rho/\epsilon$. Then the following increasing,
equi-spaced arithmetic progressions fall in the specified domains:
\begin{align*}
\bx := &\ ( 0, \frac{\rho}{m+1}, \frac{2 \rho}{m+1}, \dots,
\frac{(p-1)\rho}{m+1} ) \in \inc{[0,\epsilon)}{p}\\
\by := &\ ( \frac{-m \rho}{2m+2}, \frac{-(m-2)\rho}{2m+2}, \dots,
\frac{(m + 2p-2)\rho}{2m+2}) \in \inc{(-\rho/2, (\rho/2)+\epsilon)}{m+p}.
\end{align*}

Hence the matrix $T_\Lambda[ \bx; \by ]$ is $\TN$; reversing the
order of the rows and columns, the matrix
\[
A_m := \begin{pmatrix}
a_0 & a_1 & \cdots & & \cdots & a_m & 0 & 0 & \cdots & 0\\
0 & a_0 & \cdots & & \cdots & a_{m-1} & a_m & 0 & \cdots & 0\\
\vdots & \vdots & & & & \vdots & \vdots & \vdots & & \vdots\\
0 & 0 & \cdots & a_0 & \cdots & a_{m-p+1} & a_{m-p+2} & a_{m-p+3} & \cdots
& a_m
\end{pmatrix}_{p \times (m+p)}
\]
is $\TN$, where we define $a_\nu := \Lambda \left( \frac{(2\nu-m)
\rho}{2m+2} \right) > 0$ for $\nu = 0, 1, \dots, m$.

Once this matrix is constructed, repeat the proof of~\cite[Theorem
1]{Schoenberg55}.\footnote{This proof can be found in Karlin's book --
see \cite[Chap.~8, Theorem 3.1]{Karlin} -- and uses the
variation-diminishing property of the $\TN$ matrix $A_m$, as shown by
Schoenberg~\cite{Schoenberg30}.} This shows the polynomial
\[
p_m(z) := \frac{\rho}{m+1} \sum_{\nu=0}^m \Lambda((2\nu-m)\rho / (2m+2))
z^\nu
\]
has no roots in the sector $|\arg(z)| < p \pi / (m+p-1)$.

Now given $s \in \mathbb{C}$ and $m \geq 1$, let $z = e^{-s \rho/(m+1)}$,
and consider the holomorphic function
\[
F_m(s) := \frac{\rho}{m+1} \sum_{\nu=0}^m e^{-s (2\nu-m)\rho / (2m+2)}
\Lambda((2\nu-m)\rho / (2m+2)) = e^{sm\rho/(2m+2)} p_m(z), \qquad s \in
\mathbb{C}.
\]

From above, $F_m(s)$ has no zeros in the strip
\[
|\Im(s)| < \frac{p \pi (m+1)}{\rho(m+p-1)} = \frac{p \pi}{\rho} \left( 1
- \frac{p-2}{m+p-1} \right)
\]
for all $m$ sufficiently large. If $p=2$ then this concludes the proof;
else fixing $\delta \in (0,p\pi/\rho)$, $F_m$ has no zeros $s$
satisfying: $|\Im(s)| < (p\pi/\rho) - \delta$. By Lemma~\ref{Ltn2}, the
holomorphic Riemann sums $F_m(s)$ converge to $\mathcal{B} \{ \Lambda
\}(s)$ uniformly on each bounded domain, so by Hurwitz's theorem,
$\mathcal{B} \{ \Lambda \} \not\equiv 0$ also has no root $s$ with
$|\Im(s)| < (p\pi/\rho) - \delta$. As this holds for all $\delta \in
(0, p\pi/\rho)$, the proof is complete.
\end{proof}

\begin{remark}\label{Rschoenberg}
As noted following Theorem~\ref{ThmD}, the hypotheses therein require
using that the restriction of $\Lambda$ to the interval $I(\epsilon) :=
(-(\rho/2)-\epsilon, (\rho/2)+\epsilon)$ is $\TN_p$. If this can be
strengthened to using only $I(0) = (-\rho/2, \rho/2)$, then this would
answer Question~\ref{Qconj} in the affirmative, by specializing to
$\Lambda = W$, $\rho = \pi$, and translating from $T_W$ to $\kjk$ via
$\arctan$ as above.
Indeed, the above strengthening would imply that
the following function has no roots $s$ with $|\Im(s)| <
p$:
\[
\mathcal{B} \{ W^\alpha \}(s) = \int_{-\pi/2}^{\pi/2} e^{-sx}
\cos(x)^\alpha\ dx.
\]
Since $\alpha \in [0,\infty)$ here, the right-hand side can be computed
using a well-known, classical formula of Cauchy~\cite[pp.~40]{Cauchy},
or directly as in~\cite[\S 10]{Schoenberg55}, to yield:
\[
\mathcal{B} \{ W^\alpha \}(s) = \frac{\pi \Gamma(\alpha + 1)}{2^\alpha
\Gamma(\frac{1}{2} (\alpha + 2 + si)) \Gamma(\frac{1}{2} (\alpha + 2 -
si))},
\]
and this has roots at $s = \pm (\alpha+2) i$. It follows that $\alpha + 2
= |\alpha + 2| \geq p$. This also explains how Schoenberg's work
\cite{Schoenberg51,Schoenberg55} implies that $T_{W^\alpha}$ is not
$\TN_p$ for $\alpha \in (0,p-2)$.
\end{remark}
%}}}

%{{{1 Section 7 - Thoerem~\ref{ThmE}: Characterizing $\TN_p$ functions;
%classifying discontinuous PF/$\TN$ functions
\section{Theorem~\ref{ThmE}: Characterizing $\TN_p$ functions;
classifying discontinuous PF/$\TN$ functions}

Finally, we come to Theorem~\ref{ThmE} and a few related variants, which
characterize not only $\TN_p$ functions $\Lambda : \R \to \R$, but also
$\TN_p$ kernels $K : X \times Y \to (0,\infty)$ for general $X,Y \subset
\R$.

\subsection{Clarifications in the literature: discontinuous PF functions}

We begin by addressing some gaps found in the literature, vis-a-vis
$\TN_3$ functions and discontinuous P\'olya frequency functions. Recall
that a characterization of $\TN_p$ functions is known for $p=2$ by
Schoenberg~\cite{Schoenberg51} (see Lemma~\ref{Ltn2}). For $p=3$ an
analogous result can be found in Weinberger's work~\cite{Weinberger83},
but it turns out to have a small gap, owing to the following lemma.

\begin{lemma}\label{Llambdad}
For all $d \in [0,1]$, the following `Heaviside' function is $\TN$,
whence $\TN_3$:
\begin{equation}\label{Etn3}
H_d(x) = \begin{cases}
0, \quad & x < 0,\\
d, & x = 0,\\
1, & x > 0.
\end{cases}
\end{equation}
In particular, the function $\lambda_d(x) := e^{-x} H_d(x)$ is a P\'olya
frequency function.
\end{lemma}

Weinberger's result~\cite[Theorem 1]{Weinberger83} asserts in particular
that if $f : \R \to \R$ is $\TN_3$, then either $f(x) = H_1(ax+b)
e^{cx+c'}$ for suitable scalars $a,b,c,c' \in \R$, or the nonzero-locus
of $f$ is an open interval. However, $H_d, \lambda_d$ are nonzero on
$[0,\infty)$ and are $\TN$ for $d \in (0,1)$ as well.

\begin{remark}
In fact, this gap in~\cite{Weinberger83} stems from earlier works. In the
1947--48 announcements~\cite{Schoenberg47,Schoenberg48} of his
forthcoming results on P\'olya frequency functions, Schoenberg asserts
that $\lambda_1$ is the only discontinuous PF function, up to changes in
scale and origin.
In his full paper, in \cite[Corollary 2]{Schoenberg51}, Schoenberg
repeats this, by remarking that the only discontinuous P\'olya frequency
function is ``essentially equivalent to'' $\lambda(x) = e^{-x} {\bf 1}_{x
\geq 0}$. In particular, it seems Schoenberg was not aware of $\lambda_d$
for $d \in (0,1)$; similarly, we could not find $H_d, \lambda_d$ in the
text of Karlin~\cite{Karlin}.
\end{remark}

To our knowledge, the functions $\lambda_d$ were very recently observed
to be P\'olya frequency functions -- in joint work~\cite{BGKP-TN}, where
Lemma~\ref{Llambdad} was stated and used without a proof. Thus, in the
interest of future clarity, we quickly record a proof.

\begin{proof}[Proof of Lemma~\ref{Llambdad}]
Let $p \geq 1$ and $\bx, \by \in \inc{\R}{p}$; define $M := T_{H_d}[ \bx;
\by ]$. We prove that $\det M \geq 0$ by induction on $p$. The base case
$p=1$ is clear; for the induction step, assume $p \geq 2$ and consider
various sub-cases:
\begin{enumerate}
\item If $x_1 < y_2$, then all entries in the first row vanish, except at
most the first entry. Hence,
\[
\det T_{H_d}[ \bx; \by ] = H_d(x_1 - y_1) \det T_{H_d}[ \bx'; \by' ],
\qquad
\text{where } \bx' = (x_2, \dots, x_p), \ \by' = (y_2, \dots, y_p).
\]
Now the induction hypothesis implies $\det T_{H_d}[ \bx; \by ] \geq 0$.

\item Otherwise, suppose henceforth that $y_1 < y_2 \leq x_1$. First
suppose $y_2 = x_1$; subtracting the second row of the matrix $M$ from
the first yields a matrix with first column $(1-d, 0, \dots, 0)^T$. Now
expand along the first column and use the induction hypothesis.

\item Finally, if $y_1 < y_2 < x_1$, then the first two columns of
$T_{H_d}[ \bx; \by ]$ are identical, so $\det M = 0$.
\end{enumerate}
Finally, given any $\TN_p$ function $f(x)$ for $p \geq 1$, and scalars
$a,b \in \R$, the function $e^{ax+b} f(x)$ is also $\TN_p$, since for all
$1 \leq r \leq p$ and $\bx, \by \in \inc{\R}{r}$, the matrix
\[
T_{e^{ax+b} f}[ \bx; \by ] = D \cdot T_f[ \bx; \by ] \cdot D',
\]
where $D,D'$ are diagonal $r \times r$ matrices with $(j,j)$ entries
$e^{ax_j + b}$ and $e^{-ay_j}$ respectively. In particular, the matrix on
the left again has non-negative determinant. Hence $\lambda_d$ is also
$\TN$.
\end{proof}

We next fix the aforementioned results of Schoenberg on discontinuous
P\'olya frequency functions:

\begin{theorem}[Classification of discontinuous P\'olya frequency and
$\TN$ functions]\hfill

A P\'olya frequency function is discontinuous if and only if, up to a
change in scale and origin, it is of the form $\lambda_d$, where $d \in
[0,1]$.

More generally, a $\TN$ function $\Lambda : \R \to \R$ is discontinuous
if and only if it is the Dirac function ${\bf 1}_{x=0}$, or it is of the
form $e^{ax+b} \lambda_d(x)$ for $a,b \in \R$ and $d \in [0,1]$ -- once
again, up to a change in scale and origin.
\end{theorem}

\begin{proof}
This is shown using two results of Schoenberg from~\cite{Schoenberg51}:
\begin{itemize}
\item His classification of the non-smooth P\'olya frequency functions,
as those one-sided PF functions $\Lambda$ whose bilateral Laplace
transform has reciprocal $(1 + a_1 s) \cdots (1 + a_m s) e^{\delta s}$
with $a_m, \delta \geq 0$ and $\delta + \sum_j a_j > 0$ (and $m > 0$).
Schoenberg also shows that if $m>1$ then $\Lambda$ is continuous. Thus,
it reduces to understanding the inverse Laplace transforms of $1 /
(1+s)$. This determines $\Lambda$ via the property that a PF function is
continuous on the interior of the interval where it is positive (via
Lemma~\ref{Ltn2}). Thus $\Lambda \equiv 0$ on $(-\infty,0)$ and
$\Lambda(x) = e^{-x}$ on $(0,\infty)$. The value at the origin can be $d
\in [0,1]$, by Lemma~\ref{Llambdad}; it cannot be negative; and it is at
most $1$ by considering the $2 \times 2$ minor
\[
0 \leq \det T_\Lambda[ (1,2); (0,1) ] = \det \begin{pmatrix}
e^{-1} & \Lambda(0) \\ e^{-2} & e^{-1} \end{pmatrix} \quad \implies \quad
\Lambda(0) \leq 1.
\]

\item Schoenberg also shows that every $\TN$ function is either a Dirac
function or a P\'olya frequency function, up to an exponential factor
$e^{ax+b}$ with $a,b \in \R$.\qedhere
\end{itemize}
\end{proof}

We now make a brief digression into another assertion by Weinberger: he
used his proposed characterization of $\TN_3$ functions to show that
every power $x^\alpha$ for $\alpha \geq 1$ preserves the $\TN_3$
functions, just as every power $\alpha \geq 0$ preservers the $\TN_2$
functions. (This latter assertion is obvious from Lemma~\ref{Ltn2}, and
on any interval, not just $\R$.) In light of the above gap
in~\cite{Weinberger83}, we provide an alternate proof of this latter
result:

\begin{prop}\label{Ptp3powers}
Suppose $\alpha \in \R$. Then $x^\alpha$ preserves the class of $\TN_3$
functions if and only if $\alpha \geq 1$. (Here we use $0^0 := 0$.)
\end{prop}

If we instead use $0^0 := 1$, then clearly $x^0$ preserves the class of
$\TN_3$ functions.

\begin{proof}
First recall that the Gaussian kernel $G_1(x) := e^{-x^2}$ is a P\'olya
frequency function, whence $\TN_3$. (In fact it is $\TP$; see the proof
of Corollary~\ref{Cfekete} in the Appendix.) Examining any `principal $2
\times 2$ submatrix' of the associated kernel $T_{G_1^\alpha}$ shows that
$\alpha \geq 0$ if $T_{G_1^\alpha}$ is $\TN_2$. Now say $\alpha \in
[0,1)$. By Theorem~\ref{Tschoenberg}, $W(x)$ is $\TN_3$, but $W^\alpha$
is not $\TN_3$ (as can be directly inspected by looking at the principal
submatrix drawn at $(-\pi/4, 0, \pi/4)$, with $0^0 := 0$).

This shows one implication. Conversely, suppose $\alpha \geq 1$, and $f$
is $\TN_3$. Fix $\bx, \by \in \inc{\R}{3}$; then the matrix $T_f[ \bx;
\by ]$ is $\TN$. By Whitney's density theorem~\cite{Whi}, there is a
sequence of $3 \times 3$ $\TP$ matrices $A_k$ converging entrywise to
$T_f[ \bx; \by ]$. By \cite[Theorem 5.2]{FJS}, $A_k^{\circ \alpha}$ is
$\TP$ for all $k \geq 1$ (since $\alpha \geq 1$); now taking limits,
$T_{f^\alpha}[ \bx; \by ]$ is $\TN$ as desired.
\end{proof}

\subsection{Characterizing $\TN_p$ functions and kernels}

Returning to the above attempt by Weinberger to characterize $\TN_3$
functions, and the preceding result by Schoenberg for $\TN_2$ functions:
we now prove the aforementioned Theorem~\ref{ThmE}, characterizing
$\TN_p$ functions for all $p \geq 3$. To our knowledge there are no other
such results for $\TN_p$ functions in the literature, prior to
Theorem~\ref{ThmE}. This result will follow from a more general
formulation:

\begin{prop}\label{PFKK}
Let $t_*, \rho \in \R$ and fix a subset $Y \subset \R$ that is not
bounded above. Suppose $X \subset \R$ contains $t_* + y$ for all $\rho <
y \in Y$. Let $\Lambda : X - Y \to [0,\infty)$ be such that $\Lambda(t_*)
> 0$ and
\[
\lim_{y \in Y, \ \rho < y \to \infty} \Lambda(x_0 - y) \Lambda(t_* + y -
y_0) \to 0, \qquad \forall x_0 \in X, \ y_0 \in Y.
\]
If $\det T_\Lambda[ \bx; \by ] \geq 0$ for all $\bx \in \inc{X}{p},
\by \in \inc{Y}{p}$, then the kernel $T_\Lambda$ is $\TN_p$.
\end{prop}

Proposition~\ref{PFKK} extends a recent result of
F\"orster--Kieburg--K\"osters~\cite{FKK} in two ways: first, it works
over a large class of domains $X,Y \subset \R$, whereas the result
in~\cite{FKK} requires $X = Y = \R$. Second, even assuming $X = Y = \R$,
the result in~\cite{FKK} requires $\Lambda$ to be integrable; however,
Proposition~\ref{PFKK} is strictly more general as it works for all
$\TN_p$ functions, such as (via Remark~\ref{Rexpmod})
\begin{equation}\label{Eexpmod}
\Lambda(x) = \begin{cases}
c e^{\beta (x - x_0)}, \qquad &\text{if } x \leq x_0,\\
c e^{\alpha (x - x_0)}, &\text{if } x > x_0,
\end{cases} \qquad \text{where } -\infty \leq \alpha < \beta \leq
+\infty, \ c > 0.
\end{equation}
If now $\alpha \beta \geq 0$, then $\Lambda$ is not integrable, but the
hypotheses in Proposition~\ref{PFKK} are satisfied.

\begin{proof}[Proof of Proposition~\ref{PFKK}]
We show by downward induction on $1 \leq r \leq p$ that all $r \times r$
minors of $T_\Lambda$ on $X \times Y$ are non-negative. The $r=p$ case is
obvious, and it suffices to deduce from it the $r = p-1$ case. Thus, fix
$\bx' \in \inc{X}{p-1}$ and $\by' \in \inc{Y}{p-1}$. We are to show that
\[
\psi(x_p,y_p) := \det T_\Lambda[ (\bx', x_p); (\by', y_p) ] \geq 0\
\forall x_p > x_{p-1}, y_p > y_{p-1} \quad \implies \quad \det T_\Lambda[
\bx'; \by' ] \geq 0.
\]

We now refine the argument in~\cite{FKK}. Begin by defining the $(p-1)
\times (p-1)$ matrix $A := T_\Lambda[ \bx'; \by' ]$, and let $A^{(j,k)}$
denote the submatrix obtained by removing the $j$th row and $k$th column
of $A$. (Since $p \geq 3$, these matrices are at least $1 \times 1$.) Now
the following scalar does not depend on $x_p, y_p$:
\begin{equation}\label{EL}
L := \max_{1 \leq j,k \leq p-1} | \det A^{(j,k)} | \geq 0.
\end{equation}

Next, define $t_m \in Y$ for all $m \geq 1$ such that $t_m > \max \{
x_{p-1} - t_*, y_{p-1}, \rho \}$ and 
\[
\Lambda(x_j - t_m) \Lambda(t_* + t_m - y_k) < 1/m, \qquad \forall 0 < j,k
< p.
\]

With these choices made, we turn to the proof. Begin by expanding
$\psi(x_p, y_p)$ along the final row, and excluding the cofactor for
$(p,p)$, expand all other cofactors along the final column, to get:
\[
\psi(x_p, y_p) = \Lambda(x_p - y_p) \det(A) + \sum_{j,k=1}^{p-1}
(-1)^{j+k-1} \Lambda(x_j - y_p) \Lambda(x_p - y_k) \det A^{(j,k)}.
\]

Define $y_p^{(m)} := t_m$ and $x_p^{(m)} := t_* + t_m$, with $t_*, t_m$
as above. Then
\[
x_p^{(m)} \in X, \quad x_p^{(m)} > x_{p-1}, \quad
y_p^{(m)} \in Y, \quad y_p^{(m)} > y_{p-1}.
\]
Moreover, since $\psi(x_p^{(m)}, y_p^{(m)}) \geq 0$, we compute for $m
\geq 1$:
\[
\Lambda(t_*) \det(A) \geq \psi(x_p^{(m)}, y_p^{(m)}) - L
\sum_{j,k=1}^{p-1} \Lambda(x_j - y_p^{(m)}) \Lambda(x_p^{(m)} - y_k) \geq
- L \frac{(p-1)^2}{m}.
\]
Now taking $m \to \infty$ concludes the proof, since $\Lambda(t_*) > 0$
by assumption.
\end{proof}

\begin{remark}
Proposition~\ref{PFKK} specializes to $X = Y = G$, an arbitrary additive
subgroup of $(\R, +)$. E.g.~for $G = \Z$, we obtain a result -- whence a
characterization, akin to Theorem~\ref{ThmE} and results below -- for
`P\'olya frequency sequences of order $p$' that vanish at $\pm \infty$.
Here, $t_*$ would be an integer.
\end{remark}

With Proposition~\ref{PFKK} at hand, the final outstanding proof follows.

\begin{proof}[Proof of Theorem~\ref{ThmE}]
If $\Lambda \equiv 0$ then the result is immediate. If $\Lambda(x) =
e^{ax+b}$ then the result is again easy, since by the argument to show
Lemma~\ref{Llambdad}, it suffices to show the case of $a=b=0$, which is
obvious. Now suppose $\Lambda$ is not of the form $c e^{ax}$ for $a \in
\R$ and $c \geq 0$. Then~(2) follows by Proposition~\ref{PFKK} with
arbitrary $\rho \in \R$.

Conversely, suppose $\Lambda$ is not of the form $c e^{ax}$ for $a \in
\R$ and $c \geq 0$. Since it is $\TN_p$, clearly (1)(a),(c) follow. In
particular, since $\Lambda$ is also $\TN_2$, $g(x) := \log \Lambda(x)$ is
concave on $\R$ (in the generalized sense, i.e., it is allowed to take
the value $-\infty$), by Lemma~\ref{Ltn2}. Now let $I$ be the
nonzero-locus of $\Lambda$. If $I$ is not all of $\R$, then~(1)(b) is
immediate. If instead $\Lambda(x) > 0$ for all $x \in \R$, then since
$\Lambda$ is not an exponential, $g(x)$ is not linear from above. Hence a
short argument of Schoenberg~\cite{Schoenberg51} shows that there exist
$\beta,\gamma \in \R$ and $\delta > 0$ such that\footnote{Since $g$ is
concave, $g'$ exists and is non-increasing on a co-countable subset of
$\R$. Since $g'$ is not constant, there exist scalars $x_- < x_+$ and
$c_\pm$ such that $g'(x_-) > g'(x_+)$ and $\log \Lambda(x) \leq g'(x_\pm)
x + c_\pm$. Choose $\gamma, \delta \in \R$ such that $g'(x_+) < \gamma -
\delta < \gamma + \delta < g'(x_-)$. Then $\log \Lambda(x) - \gamma x$ is
bounded above by $(g'(x_\pm) - \gamma)x + c_\pm$, for $\pm x > 0$.}
\[
e^{-\gamma x} \Lambda(x) \leq e^{\beta-\delta |x|}, \qquad \text{as } x
\to \pm \infty.
\]
From this, the decay property~(1)(b) immediately follows.
\end{proof}

We conclude by extending the above result to arbitrary positive-valued
kernels on $X \times Y$:

\begin{prop}
Let $X,Y \subset \R$ be non-empty, and $K : X \times Y \to (0,\infty)$ a
kernel satisfying any of the following decay conditions:
\begin{align*}
\sup Y \not\in Y, \qquad & \lim_{y \in Y, \ y \to (\sup Y)^-} K(x_0,y) =
0, \qquad \forall x_0 \in X,\\
\inf Y \not\in Y, \qquad & \lim_{y \in Y, \ y \to (\inf Y)^+} K(x_0,y) =
0, \qquad \forall x_0 \in X,\\
\sup X \not\in X, \qquad & \lim_{x \in X, \ x \to (\sup X)^-} K(x,y_0) =
0, \qquad \forall y_0 \in Y,\\
\inf X \not\in X, \qquad & \lim_{x \in X, \ x \to (\inf X)^+} K(x,y_0) =
0, \qquad \forall y_0 \in Y.
\end{align*}
Given an integer $p \geq 2$, 
the kernel $K$ is $\TN_p$ on $X \times Y$, if and only if
every $p \times p$ minor of $K$ is non-negative.
\end{prop}

For instance, this can be specialized to kernels over $X = Y = G$, an
additive subgroup of $(\R,+)$.

\begin{proof}
One implication is immediate. Conversely, as in the preceding proofs it
suffices to show that $K[ \bx'; \by' ] \geq 0$ for all tuples $\bx' \in
\inc{X}{p-1}, \by' \in \inc{Y}{p-1}$. We show this under the fourth decay
condition; the other cases are similar to this proof and the proofs
above. Fix increasing tuples
\[
\bx' := (x_2, \dots, x_p) \in \inc{X}{p-1}, \quad \by' := (y_2, \dots,
y_p) \in \inc{Y}{p-1}
\]
as well as $y_1 \in (-\infty,y_2) \cap Y$. Let $A = K[ \bx'; \by' ]$ and
define $L \geq 0$ as in~\eqref{EL} above. Also choose for each $m \geq1$
an element $x_1^{(m)} \in X$, such that $x_1^{(m)} < x_2$ and
$K(x_1^{(m)},y_k) < 1/m$ for $2 \leq k \leq p$.
Now compute as in the proof of Proposition~\ref{PFKK}, this time
expanding the determinant along the first row and column:
\begin{align*}
K(x_1^{(m)}, y_1) \det(A) \geq &\ \det K[ (x_1^{(m)}, \bx'); (y_1, \by')
] - L \sum_{j,k=2}^p K(x_j, y_1) K(x_1^{(m)}, y_k)\\
\geq &\ \det K[ (x_1^{(m)}, \bx'); (y_1, \by') ] - \frac{L(p-1)}{m}
\sum_{j=2}^p K(x_j, y_1).
\end{align*}
As $\det K[ (x_1^{(m)}, \bx'); (y_1, \by') ] \geq 0$ and
$K(x_1^{(m)},y_1) > 0$, the result follows by letting $m \to \infty$.
\end{proof}

\begin{remark}
We have tried to keep the proofs of the results in our main theme
self-contained (modulo the Appendix) -- specifically, for the results
related to powers preserving $\TN_p$. The only four such proofs that use
prior results are those of Corollary~\ref{C33}, Theorems~\ref{ThmB};
Theorem~\ref{ThmC}(1); and Theorem~\ref{ThmD}, which use Schoenberg's
characterization of PF functions~\cite{Schoenberg51}; Lemma~\ref{Lfekete}
(Fekete); Theorem~\ref{Tschoenberg} (Schoenberg); and Schoenberg's
\cite[Theorem 1]{Schoenberg55} plus Sierpi\'nsky's
result~\cite{Sierpinsky}, respectively.
\end{remark}
%}}}

\subsection*{Acknowledgments}

This work is partially supported by
Ramanujan Fellowship grant SB/S2/RJN-121/2017,
MATRICS grant MTR/2017/000295, and
SwarnaJayanti Fellowship grants SB/SJF/2019-20/14 and DST/SJF/MS/2019/3
from SERB and DST (Govt.~of India),
and by grant F.510/25/CAS-II/2018(SAP-I) from UGC (Govt.~of India).

%{{{1 Bibliography

%}}}

%{{{1 Appendix - Self-contained proofs from previous papers
\appendix
\section{Proofs from previous papers}

In the interest of keeping this paper as self-contained as possible, this
Appendix contains short proofs (from the original papers) of the results
which are stated above and are used in proving our main theorems. The
reader is welcome to skip these proofs (certainly in a first reading).

\begin{proof}[Proof of Theorem~\ref{Tfitzhorn}(1)]
We show the `if' part; the converse was shown in the proof of
Theorem~\ref{Tjain}(1). If $\alpha \in \Z^{\geq 0}$ then $x^\alpha$
preserves Loewner positivity by the Schur product
theorem~\cite{Schur1911}. If $\alpha \geq n-2$, we show the result by
induction on $n \geq 2$, with the $n=2$ case obvious. Suppose $n \geq 3$
and $A \in \bp_n((0,\infty))$. Let $\zeta$ denote the last column of $A$,
and $B := a_{nn}^{-1}\zeta \zeta^T$. Then $B \geq 0$; moreover, $A-B$ has
last row and column zero, and is itself positive semidefinite via Schur
complements. Now FitzGerald--Horn employ a useful `integration trick': by
the Fundamental Theorem of Calculus,
\[
A^{\circ \alpha} = B^{\circ \alpha} + \alpha \int_0^1 (A-B) \circ
(\lambda A + (1-\lambda)B)^{\circ (\alpha-1)}\ d \lambda.
\]
But $A-B$ has last row/column zero, and the leading principal $(n-1)
\times (n-1)$ submatrix of the integrand is in $\bp_{n-1}(\R)$ by the
induction hypothesis. We are done by induction.
\end{proof}

\begin{proof}[Proof of Theorem~\ref{Tkarlin} for integer powers]
For integers $\alpha \geq 0$, the proof that $x^\alpha e^{-\alpha x} {\bf
1}_{x \geq 0}$ is a P\'olya frequency function is in steps. We first show
that the kernel $K(x,y) := {\bf 1}_{x \geq y}$ is $\TN$ on $\R \times
\R$. This is a direct calculation; e.g., Karlin~\cite[pp.~16]{Karlin}
checks for the `transpose' kernel $K(x,y) := {\bf 1}_{x \leq y}$:
\[
\det K[ \bx; \by ] = {\bf 1}(x_1 \leq y_1 < x_2 \leq y_2 < \cdots < x_p
\leq y_p),
\]
for all $p \geq 1$ and tuples $\bx, \by \in \inc{\R}{p}$. (Alternately,
use Lemma~\ref{Llambdad}.) Now pre- and post-multiplying with diagonal
matrices with $(k,k)$ entries $e^{-x_k}$ and $e^{y_k}$ respectively,
shows that the kernel $\Omega_0(x) := e^{-x} {\bf 1}_{x \geq 0}$ is a
P\'olya frequency function. Next, the `Basic Composition Formula' of
P\'olya--Szeg\"o (see e.g.~\cite[pp.~17]{Karlin}) shows that the class of
P\'olya frequency functions is closed under convolution. But for any
integer $\alpha \geq 1$, the $\alpha$-fold convolution of $\Omega_0(x)$
with itself, yields precisely $x^{\alpha-1} e^{-x} {\bf 1}_{x \geq 0}$.
Finally, multiplying with a suitable exponential function shows
$\Omega^\alpha$ is still integrable, so also a P\'olya frequency
function.
\end{proof}

\begin{remark}\label{Rexpmod}
Let $\Lambda(x)$ be as in~\eqref{Eexpmod}.
If $|\alpha|$ or $|\beta|$ is infinite, $\Lambda$ equals $\lambda_0$ or
$\lambda_1$ (up to a linear change of variables), hence is $\TN$.
Else if $\alpha = \beta$ then $\Lambda$ is an exponential -- up to
rescaling -- so any submatrix drawn from has rank one, whence $\Lambda$
is $\TN$. Finally, suppose $\alpha < \beta \in \R$. As explained in
Lemma~\ref{Llambdad}, $\lambda_1(x) = e^{-x} {\bf 1}_{x \geq 0}$ is
$\TN$, whence so is $\lambda_1(-x)$. As in the preceding proof, the Basic
Composition Formula implies that $\lambda_1(x) \ast \lambda_1(-x) =
e^{-|x|} / 2$ is also $\TN$. By a linear change of variables, the
function $e^{(\alpha - \beta)|x|/2}$ is $\TN$. Multiplying by $e^{(\alpha
+ \beta)x/2}$, the function in~\eqref{Eexpmod} is also $\TN$.
\end{remark}

\begin{proof}[Proof of Proposition~\ref{Pjain}]
This Descartes-type result is proved in the spirit of Laguerre and
Poulain's classical arguments, via Rolle's theorem.
In this proof-sketch, we also address a small gap in~\cite{Jain2}. The
first step is to observe that $1 + u x_j > 0$ for all $j$ if and only if
$u \in (A_\bx, B_\bx)$. Also note that
\begin{equation}\label{Ejain}
A_{-\bx} = - B_\bx, \quad \text{and} \quad A_\bx < 0 < B_\bx, \qquad
\forall \bx \in \R^n.
\end{equation}

We now sketch the proof in~\cite{Jain2}. If $r=0$ then the result is
immediate, so we suppose henceforth that $r \neq 0$. Denote by $s \leq
n-1$ the number of sign changes in $\bc$ after removing the zero
coordinates. We then claim that the number of zeros is at most $s$; the
proof is by induction on $n \geq 1$ and then on $s \geq 0$. The base
cases of $n=1$, and $s=0$ for any $n \geq 1$, are easy to show. For the
induction step, we may suppose all $c_j$ are non-zero, and the $x_j$ are
in increasing order.

The first case is that whenever there is a sign change in $\bc$,
i.e.~$c_{k-1} c_k < 0$, we always have $x_k \leq 0$. (This is a small
clarification that was not addressed in~\cite{Jain2}; on a related
note,~\eqref{Ejain} does not appear there.) In this case we simply
replace $\bx$ by $-\bx$ and $\bc$ by $\bc' := (c_n, \dots, c_1)$. So the
assertion for $\varphi_{-\bx,\bc',r} : (-B_\bx, -A_\bx) \to \R$
(via~\eqref{Ejain}) would show the result for $\varphi_{\bx,\bc,r}$.

Thus there exists $k$ with $c_{k-1} c_k < 0 < x_k$. In turn, there exists
$v > 0$ with $1 - v x_k < 0 < 1 - v x_{k-1}$, so that the sequence $c_j
(1 - v x_j)$, $j=1,\dots,n$ has one less sign change than ${\bf c}$. Now
define
\[
\psi(u) := \sum_{j=1}^n c_j (1 - v x_j) (1 + u x_j)^{r-1}, \qquad h(u) :=
(u+v)^{-r} \varphi_{\bx,\bc,r}(u), \qquad u \in (A_\bx, B_\bx),
\]
so the induction hypothesis applies to $\psi$. But a straightforward
computation yields
\[
\psi(u) = \frac{-(u+v)^{r+1}}{r} h'(u), \quad \text{and} \quad u+v > 0,
\qquad \forall u \in (A_\bx, B_\bx),
\]
so by the induction hypothesis, $h'$ has at most $s-1$ zeros. We are done
by Rolle's theorem.
\end{proof}

\begin{proof}[Proof-sketch of Proposition~\ref{Pjain2}]
Suppose $\alpha \in \R \setminus \{ 0, 1, \dots, n-2 \}$, and $S^{\circ
\alpha} \bc^T = 0$ for a tuple $\bc = (c_1, \dots, c_n) \neq 0$.
Rewriting this in the language of Proposition~\ref{Pjain} yields:
\[
\varphi_{\bx,\bc,\alpha}(y_k) = \sum_{j=1}^n c_j (1 + y_k x_j)^\alpha =
0, \qquad \forall 1 \leq k \leq n.
\]
By assumption, $y_k \in (A_\bx, B_\bx)$ for all $k$ (see the line
preceding~\eqref{Ejain}), so Proposition~\ref{Pjain} implies
$\varphi_{\bx,\bc,\alpha} \equiv 0$ on $(A_\bx,B_\bx)$. By~\eqref{Ejain},
$\varphi_{\bx,\bc,\alpha}^{(k)}(0) = 0, \ \forall 0 \leq k \leq n-1$.
This system can be written as
\[
W_\bx^{(n-1)} D \bc^T = 0, \quad \text{where }
W^{(r)}_{\bf x} := \begin{pmatrix} 1 & 1 & \cdots & 1\\
x_1 & x_2 & \cdots & x_n\\
\vdots & \vdots & \ddots & \vdots\\
x_1^r & x_2^r & \cdots & x_n^r\end{pmatrix}, \ r \in \Z^{\geq 0}
\]
and $D$ is the diagonal matrix with diagonal entries $1, \alpha,
\alpha(\alpha-1), \dots, \alpha (\alpha-1) \cdots (\alpha-n+2)$. By
assumption on $\alpha$, the matrix $D$ is non-singular, as is the
(usual) Vandermonde matrix $W^{(n-1)}_\bx$. Hence $\bc = 0$, and so
$S^{\circ \alpha}$ is non-singular.

Finally, if $\alpha \in \{ 0, \dots, n-2 \}$, then
$S^{\circ \alpha} = (W_\by^{(\alpha)})^T D W_\bx^{(\alpha)}$,
where $W_\bx^{(\alpha)}$ was defined above, and $D$ is a diagonal
$(\alpha + 1) \times (\alpha+1)$ matrix with $(k,k)$ entry
$\binom{n}{k}$. Since these matrices are each of maximal possible rank,
the result follows.
\end{proof}

\begin{proof}[Proof of Corollary~\ref{Cthma}]
Here we reproduce Karlin's proof of the assertion $(2) \implies (1)$.
Since $\Lambda(x)$ is a one-sided P\'olya frequency function if and only
if $\Lambda(-x)$ is, we may assume without loss of generality that
$\Lambda(x) = 0$ for sufficiently small $x<0$. Now $\mathcal{B} \{
\Lambda \}(s)$ is of the form
\begin{equation}
\mathcal{B} \{ \Lambda \}(s) = e^{-\delta s} \prod_{j=1}^\infty (1 + a_j
s)^{-1}, \quad \text{where }
a_j \geq 0, \ \delta \in \R, \ \sum_j a_j < \infty,
\end{equation}
by foundational results of Schoenberg~\cite{Schoenberg51}.
If $\alpha \in \Z^{>0} \cup (p-1,\infty)$, and $a_j > 0$, then
\[
(1 + a_j s)^{-\alpha} = \mathcal{B} \{ \Lambda_{j,\alpha} \}(s), \qquad
\text{where } \Lambda_{j,\alpha}(x) = \frac{e^{(a_j^{-1} +
\alpha-1)x}}{\Gamma(\alpha) a_j^\alpha} \Omega(x)^{\alpha-1},
\]
where $\Omega(x)$ is Karlin's kernel from~\cite{KarlinTAMS}. By choice of
$\alpha$, we have $\alpha - 1 \in \Z^{\geq 0} \cup [p-2,\infty)$, so
$\Omega(x)^\alpha$ is $\TN_p$, whence so is $\Lambda_{j,\alpha}(x)$ by the
concluding argument in the proof of Lemma~\ref{Llambdad}. Now the
convolution of finitely many of the one-sided integrable $\TN_p$ functions
$\Lambda_{j,\alpha}, \ j \geq 1$ is still an integrable $\TN_p$ function,
by the Basic Composition Formula (see above in this Appendix).

Finally, suppose all $a_j > 0$. Since Karlin's proof of \cite[Chapter 7,
Theorem 12.2]{Karlin} does not address this case explicitly, we add a few
lines for completeness. Since $1 + a_j s \leq e^{a_j s}$ and $\sum_j a_j
< \infty$, we have
$\prod_{j \geq 1} |1 + a_j s| \leq \prod_{j \geq 1} (1 + a_j |s|) \leq
e^{|s| \sum_{j \geq 1} a_j} < \infty$.
Hence,
\[
\phi_n(s) := \frac{e^{-\delta s}}{\prod_{j=1}^n (1 + a_j s)^\alpha} \quad
\text{converges to} \quad \phi(s) := \frac{e^{-\delta s}}{\prod_{j \geq
1} (1 + a_j s)^\alpha}
\]
on the strip $\Re(s) > \max_j -a_j^{-1}$. Moreover, for $x \in \R$ the
functions $\phi_n(i x)$ are bounded above -- uniformly for all $n \geq 2
/ \alpha$ -- by an integrable function of the form $1/(1+a^2x^2)$. More
precisely,
\[
|\phi_n(ix)| \leq \prod_{j=1}^{\lceil 2/\alpha \rceil} |1 + i a_j
x|^{-\alpha}
\leq \sqrt{1 + (ax)^2}^{\; -\alpha \cdot \lceil 2/\alpha \rceil} \leq
\frac{1}{1 + (ax)^2}, \qquad \forall x \in \R, \ n \geq 2/\alpha,
\]
where $a = \min \{ a_1, \dots, a_{\lceil 2/\alpha \rceil} \} > 0$. Now
apply the Lebesgue dominated convergence theorem and repeat the argument
on~\cite[pp.~334]{Karlin}, to show that the Fourier--Mellin integrals of
$\phi_n$, which are $\TN_p$ functions vanishing on $(-\infty,\delta)$,
converge to that of $\phi$, which function therefore possesses the same
properties.
\end{proof}

\begin{proof}[Proof of Lemma~\ref{Lconvex}]
First suppose $0 \leq B \leq A$ are as claimed. For $\lambda \in (0,1)$,
the Loewner convexity condition can be reformulated in two ways:
\begin{align*}
\frac{f[B + \lambda (A-B)] - f[B]}{\lambda} \leq &\ f[A] - f[B],\\
\frac{f[A + (1-\lambda)(B-A)] - f[A]}{1-\lambda} \leq &\ f[B] - f[A].
\end{align*}
Now let $\lambda \to 0^+$ and $\lambda \to 1^-$, respectively. We obtain:
\[
(A - B) \circ f'[B] \leq f[A] - f[B], \qquad (B - A) \circ f'[A] \leq
f[B] - f[A].
\]
Summing these inequalities gives $(A - B) \circ (f'[A] - f'[B]) \geq 0$.
Since $A-B$ has only non-zero entries, it has a positive semidefinite
`Schur-inverse'. Take the Schur product with this matrix to obtain $f'[A]
\geq f'[B]$, as claimed. Adapting the same argument shows that
$f'[A_\lambda] \geq f'[A_\mu]\ \forall 0 \leq \mu \leq \lambda \leq 1$,
where $A_\lambda := \lambda A + (1-\lambda) B$.

Conversely, suppose $0 \leq B \leq A$ in $\bp_n((0,\infty))$ are
arbitrary, and $f'$ preserves Loewner monotonicity on $[B,A]$. In the
spirit of previous proofs for powers preserving Loewner positivity and
monotonicity (see above), another `integration trick' yields:
\begin{align}\label{Ehiai}
\begin{aligned}
f[ (A+B)/2 ] - f[B]  = &\ \frac{1}{2} \int_0^1 (A - B) \circ f' \left[
\lambda \frac{A+B}{2} + (1-\lambda)B \right]\ d \lambda,\\
\frac{f[A] + f[B]}{2} - f[B] = &\ \frac{f[A] - f[B]}{2} = \frac{1}{2}
\int_0^1 (A - B) \circ f' \left[ \lambda A + (1-\lambda)B \right]\ d
\lambda.
\end{aligned}
\end{align}
Using the Schur product theorem and the hypotheses on $f'$,
\[
(A - B) \circ f' \left[ \lambda \frac{A+B}{2} + (1-\lambda)B \right]
\leq
(A - B) \circ f'[\lambda A + (1-\lambda) B].
\]
Together with~\eqref{Ehiai}, this yields $f[(A+B)/2] \leq \frac{1}{2}
(f[A] + f[B])$. Now an easy induction argument, first on $m \geq 1$ and
then on $k \in [1, 2^m]$, yields
\[
f \left[ \frac{k}{2^m} A + \left(1 - \frac{k}{2^m} \right) B \right] \leq 
\frac{k}{2^m} f[A] + \left(1 - \frac{k}{2^m} \right) f[B], \qquad \forall
m \geq 1, \ 1 \leq k \leq 2^m.
\]

Finally, given $\lambda \in (0,1)$ we approximate $\lambda$ by a sequence
of dyadic rationals of the form $k / 2^m$. Now the preceding inequality
and the continuity of $f$ allows us to deduce that $f$ preserves Loewner
convexity on $\{ A, B \}$. The same arguments can be adapted, as in the
preceding half of this proof, to show that $f$ preserves Loewner
convexity on $\{ A_\lambda, A_\mu \}$ for $0 \leq \mu \leq \lambda \leq
1$.
\end{proof}

\begin{proof}[Proof of Corollary~\ref{Cfekete}]
For the `if' part, note that every contiguous minor of a Hankel matrix
$A$ is a contiguous principal minor of either $A$ or $A^{(1)}$. This
shows the result for $\TP_p$ by Fekete's lemma~\ref{Lfekete}. For
$\TN_p$, first let $B$ be a matrix drawn from the Gaussian kernel, say $B
= (e^{-(x_j-y_k)^2})_{j,k=1}^n$, with $\bx, \by \in \inc{\R}{n}$. Then $B
= D_\bx V D_\by$, where $D_\bx$ for a vector $\bx$ is the diagonal matrix
with $(k,k)$ entry $e^{-x_k^2}$, and $V$ is the generalized Vandermonde
matrix with $(j,k)$ entry $e^{2 x_j y_k} = (e^{2 x_j})^{y_k}$, whence
non-singular. As every submatrix of $B$ is of this form, it follows that
$B$ is $\TP$.

Now given $A_{n \times n}$ Hankel as specified, we have that all
contiguous minors of $A$ of order $\leq p$ are non-negative. Since the
corresponding submatrices are symmetric (and Hankel), it follows that
they are all positive semidefinite. Let $B := (e^{-(j-k)^2})_{j,k=1}^n$;
then $B$ is $\TP$ from above. It follows for $\epsilon > 0$ that every
contiguous submatrix of $A+\epsilon B$ of order $\leq p$ is positive
definite. By Fekete's result, $A+\epsilon B$ is $\TP_p$. Letting
$\epsilon \to 0^+$, $A$ is $\TN_p$.
The `only if' part follows by definition. 
\end{proof}

\begin{proof}[Proof of continuity in Proposition~\ref{Phankeltnp}]
We claim that $f \equiv 0$ or $f > 0$ on $(0,\infty)$. Indeed, suppose
$f(x_0) = 0$ for some $x_0 > 0$. Choose $0 < x < x_0 < y$, apply $f$
entrywise to the Hankel $\TN$ matrices
$\begin{pmatrix} x_0 & x\\ x & x_0 \end{pmatrix}, \
\begin{pmatrix} x_0 & y\\ y & y^2/x_0 \end{pmatrix}$,
and take determinants. It follows that $f(x) = f(y) = 0$, as desired.
Using the first of the above test matrices also shows that $f$ is
non-decreasing on $(0,\infty)$.

Now suppose $f > 0$ on $(0,\infty)$, and fix $t>0$. We present Hiai's
argument from~\cite{Hiai2009} to show $f$ is continuous at $t$. For
$\epsilon \in (0,t/5)$, we have $0 < t + \epsilon \leq \sqrt{(t + 4
\epsilon) (t - \epsilon)}$. It follows that
\[
f(t + \epsilon) \leq f \left( \sqrt{ (t + 4 \epsilon) (t - \epsilon) }
\right) \leq  \sqrt{ f(t + 4 \epsilon) f(t - \epsilon) },
\]
where the second inequality follows by taking the determinant, after
applying $f$ entrywise to the matrix
\[
\begin{pmatrix} t + 4\epsilon & \sqrt{ (t + 4 \epsilon) (t - \epsilon)
}\\ \sqrt{ (t + 4 \epsilon) (t - \epsilon) } & t - \epsilon
\end{pmatrix}.
\]
Now take $\epsilon \to 0^-$; then continuity follows, since $f$ is
positive and non-decreasing on $(0,\infty)$:
\[
0 < f(t) \leq f(t^+) \leq f(t^-) \leq f(t), \qquad \forall t > 0.
\qedhere
\]
\end{proof}

\begin{proof}[Proof of Lemma~\ref{Lweinberger}]
Let the function $M(x) = 2e^{-|x|} - e^{-2|x|}$ for $x \in \R$. For all
integers $n \geq 1$,
\[
\mathcal{B} \{ M^n \}( s ) = 2 \sum_{k = 0}^n ( -1 )^{k + 1} \binom{n}{k}
\frac{2^{n - k} ( n + k )} {s^2 - ( n + k )^2} = \frac{p_n( s )}{q_n(
s)},
\]
say, is the bilateral Laplace transform of $M(x)^n$. Here the polynomial
$q_n(s) = \prod_{k=0}^n (s^2 - (n+k)^2)$ has all simple roots, and degree
$2n+2$. It is easy to check that $\deg(p_n) \leq 2n$.

Now for $n=1$ this yields $12 / ((s^2-1)(s^2-4))$, whose reciprocal is a
polynomial, so classical results of Schoenberg~\cite{Schoenberg51} imply
that $M(x)$ is a P\'olya frequency function. Also note that $\deg(p_n)
\leq 2n$, and one checks by direct evaluation that $p_n(\pm (n+k))$ is
non-zero for $0 \leq k \leq n$, so $p_n$ does not vanish at any root $\pm
(n+k)$ of $q_n$. Finally, $p_n(n)/p_n(2n)$ is also checked to be $>1$.
Hence the rational function $q_n/p_n$ is not a polynomial for $n>1$ -- in
fact, not in the Laguerre--P\'olya class. The aforementioned results of
Schoenberg now imply that $M(x)^n$ is not a P\'olya frequency function.
As $M(x)^n$ is integrable and non-vanishing at two points, it follows
that $M(x)^n$ is not $\TN$ for $n>1$.
\end{proof}
%}}}

\end{document}